\documentclass[11pt, oneside]{article}   	
\usepackage{amsmath, amsthm, amsfonts, amssymb}
\usepackage{graphicx}
\usepackage{subcaption}

\usepackage[colorlinks,citecolor=blue,urlcolor=blue]{hyperref}
\newcommand\Ex{{\mathbb E}}

\newcommand\Prob{{\mathbb P}}

\newcommand\Normal{{\mathcal N}}

\newcommand\cA{{\mathcal A}}
\newcommand\cB{{\mathcal B}}
\newcommand\cC{{\mathcal C}}

\newcommand\cF{{\mathcal F}}

\newcommand\cK{{\mathcal K}}

\newcommand\cP{{\mathcal P}}

\newcommand\cV{{\mathcal V}}
\newcommand\cW{{\mathcal W}}
\newcommand\cX{{\mathcal X}}

\newcommand\cO{{\mathcal O}}
\newcommand\cE{{\mathcal E}}

\newcommand\N{{\mathbb N}}
\newcommand\Q{{\mathbb Q}}
\newcommand\Z{{\mathbb Z}}
\newcommand\R{{\mathbb R}}

\newcommand\E{{\mathbb E}}
\newcommand\I{{\mathbb I}}

\newcommand\Fi{{\mathbf F}}
\newcommand\bM{{\mathbf M}}
\newcommand\bN{{\mathbf N}}
\newcommand\bS{{\mathbf S}}

\newcommand\bX{{\mathbf X}}

\newcommand\dto{\overset{d}{\to }}
\newcommand\Pto{\overset{\Prob}{\to}}
\newcommand\hatPto{\overset{\hat \Prob}{\to }}

\newcommand\zero{{o}}
\newcommand\one{{\bf 1}}

\newcommand\lr[1]{\langle #1 \rangle}

\DeclareMathOperator{\Image}{Im}

\DeclareMathOperator{\Var}{Var}
\DeclareMathOperator{\Cov}{Cov}
\DeclareMathOperator{\rank}{rank}

\newtheorem{theorem}{Theorem}[section]
\newtheorem{corollary}[theorem]{Corollary}
\newtheorem{lemma}[theorem]{Lemma}
\newtheorem{proposition}[theorem]{Proposition}

\theoremstyle{definition}
\newtheorem{definition}[theorem]{Definition}

\theoremstyle{remark}
\newtheorem{remark}[theorem]{Remark}

\title{Random connection models in the thermodynamic regime: central limit theorems for add-one cost stabilizing functionals}

\author{Van Hao Can
\footnote{Department of Statistics and Applied Probability, National University of Singapore, and Institute of Mathematics, Vietnam Academy of Science and Technology.
\newline
Email: cvhao89@gmail.com
}		
\and 
Khanh Duy Trinh
\footnote{Global Center for Science and Engineering, Waseda University, Japan.
\newline
Email: trinh@aoni.waseda.jp 
} 
}

\begin{document}
\maketitle

\begin{abstract}
The paper deals with a random connection model, a random graph whose vertices are given by a homogeneous Poisson point process on $\R^d$, and edges are independently drawn with probability depending on the locations of the two end points. We establish central limit theorems (CLT) for general functionals on this graph under minimal assumptions that are a combination of  the weak stabilization for the-one cost and a $(2+\delta)$-moment condition. As a consequence, CLTs for  isomorphic subgraph counts, isomorphic component counts, the number of connected components are then derived. In addition, CLTs for Betti numbers and the size of biggest component are also proved for  the first time.

\medskip

	
\medskip
	

\end{abstract}

\section{Introduction}

Given a configuration $\cP$ of a homogeneous Poisson point process on the $d$-dimensional Euclidean space $\R^d$ with density $\lambda > 0$, and a measurable symmetric connection function $\varphi \colon \R^d \to [0,1]$, connect any two distinct points $x, y \in \cP$ with probability $\varphi(x - y)$ independently of the other pairs. The resulting random graph $G_\varphi(\cP)$ is called a random connection model (RCM) with parameters $(\lambda, \varphi)$. For a special choice of $\varphi$ that $ \varphi(x) = \I(|x| \le r)$, where $\I$ is the indicator function and $|x|$ is the Euclidean norm of $x$, $G_\varphi(\cP)$ becomes a random geometric graph, where two vertices are connected, if their distance is less than or equal to the threshold $r > 0$. Figure~\ref{f1} illustrates a RCM and a random geometric graph built on the same set of vertices. RCMs, including general models where a point process is taken in an abstract space, have been known as a very useful model with many applications in physics, epidemiology and telecommunications, see e.g.\ \cite{H1}. Therefore, they have gained a great interest from many scientists in different branches of science \cite{H4, H2,H3}. In particular, for the mathematical side, problems such as connectivity, diameter,  degree counts and the number of connected components have been studied \cite{H5, Giles-2016, Iyer-2018, Last-Nestmann-Schulte-2018}.

\begin{figure}
\centering
\begin{subfigure}[b]{0.45\textwidth}
\fbox{\includegraphics[width=\textwidth]{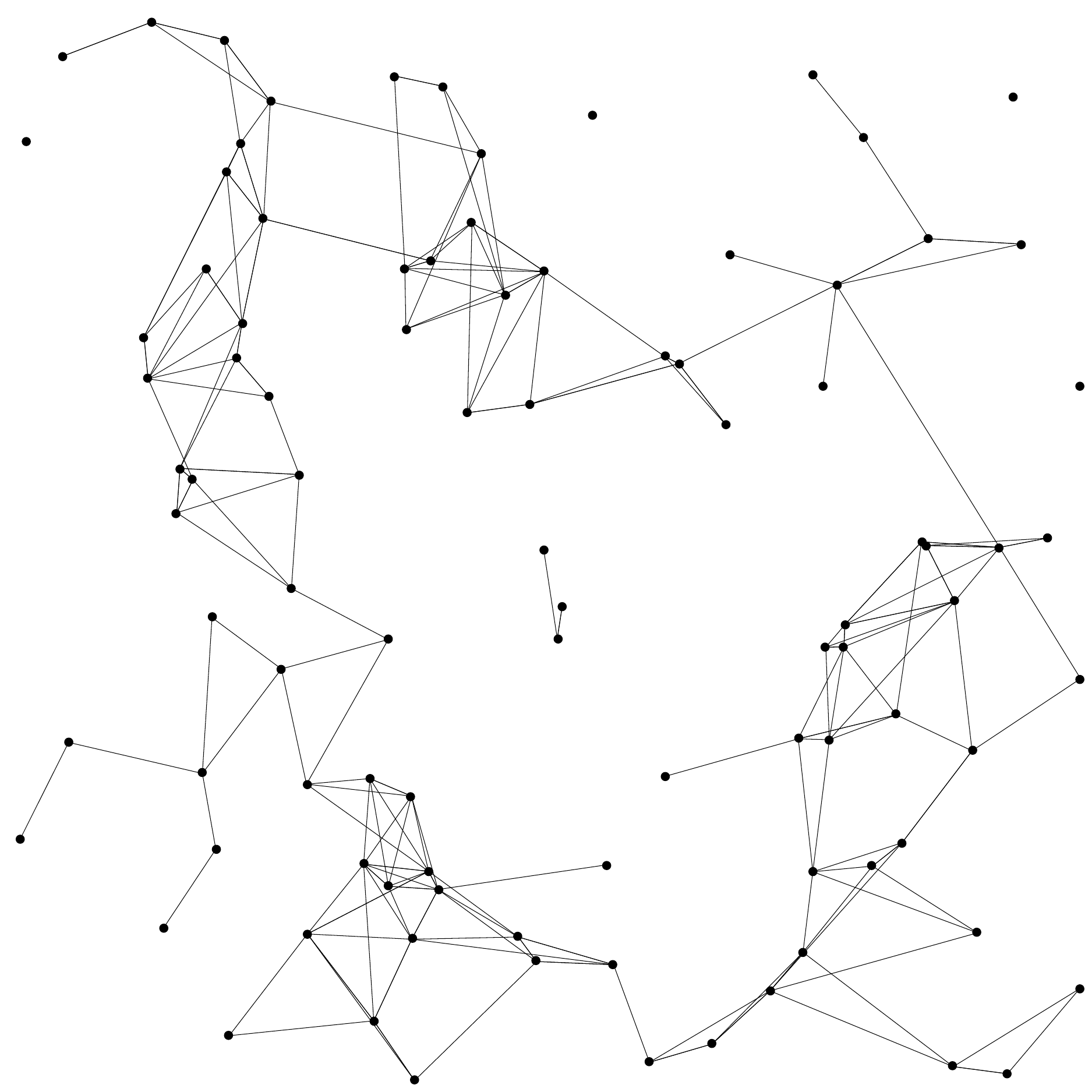}}
\caption{Random connection model}
\end{subfigure}
\quad
\begin{subfigure}[b]{0.45\textwidth}
\fbox{\includegraphics[width=\textwidth]{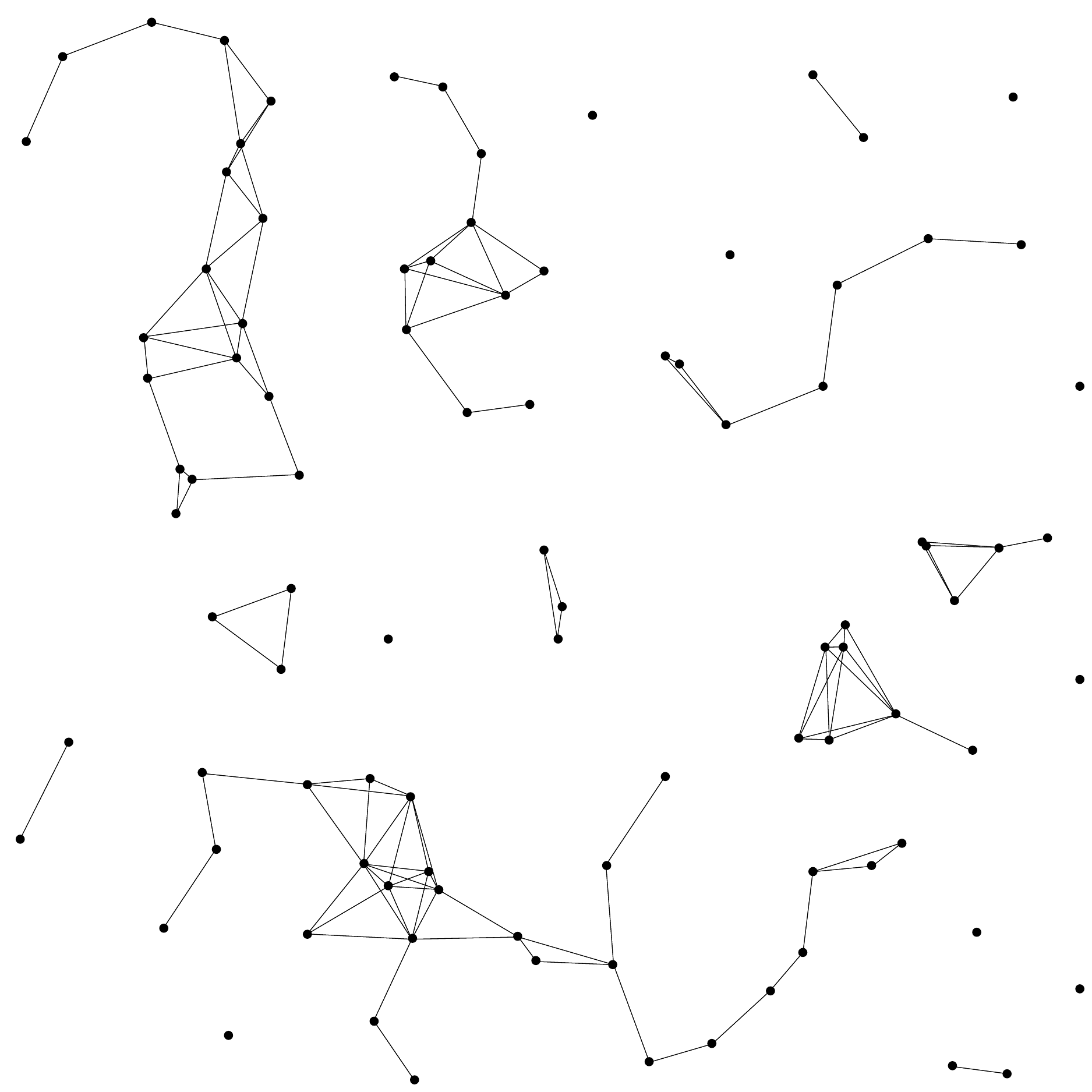}}
\caption{Random geometric graph}
\end{subfigure}
\caption{Illustration of (a) a random connection model with $\lambda = 1$ and $\varphi (x) = e^{-|x|^2}$ restricted on the rectangle $[-5,5]^2$, and (b) a random geometric graph on the same set of vertices with $r = 1$.}
\label{f1} 
\end{figure}

In this paper, we focus  on studying the  asymptotic behavior of general functionals on the RCM for fixed parameters $(\lambda, \varphi)$.  Assume that $f$ is a functional defined on finite graphs. For a bounded window $W \subset \R^d$, let $G_\varphi(\cP)|_W$ be the restriction of the random graph $G_\varphi(\cP)$ on the set of vertices lying in $W$. Our aim is to establish a central limit theorem (CLT) for $f(G_\varphi(\cP)|_W)$ as the window $W$ tends to $\R^d$. The result here generalizes CLTs in \cite{PY-2001, Trinh-2019} for stabilizing functionals on homogeneous Poisson point processes. We first extend the concept of weakly stabilization which is original from \cite{PY-2001} to this setting. Then a CLT holds under assumptions that the functional is weakly stabilizing and satisfies a moment condition. Our result should be a counterpart to a general result on normal approximation in \cite{Last-Nestmann-Schulte-2018}.

Let us introduce the result in more details. Let $G(\cP \cup \{\zero\})$ be a random graph obtained from $G(\cP) = G_\varphi(\cP)$ by adding the origin $\{\zero\}$ and edges $\{\zero, x\}, x \in \cP$ independently with probability $\varphi(x)$. We define the add-one cost of $f$ as 
\[
	D_\zero f(W) := f(G(\cP \cup \{\zero\})|_W) - f(G(\cP)|_W),  
\]
where $W$ is a bounded subset of $\R^d$. This is the cost paid by adding a point at the origin. Then the functional $f$ is said to be {\it weakly stabilizing} if there is a random variable $\Delta$, called the limit add-one cost, such that for any sequence of cubes $\{W_n\}$ tending to $\R^d$, 
\[
	D_\zero f(W_n) \to \Delta \quad \text{in probability as } n \to \infty.
\]
Here by a cube, we mean a subset of the form $\prod_i^d [x_i, x_i + l), x_i \in \R, l > 0$. Our CLT is stated as follows. 
\begin{theorem} \label{theo:1}
Assume that the functional $f$ is weakly stabilizing and satisfies the following moment condition
\[
	\sup_{\zero \in W: \text{cube}} \Ex[|D_\zero f(W)|^p] < \infty,
\]
for some $p > 2$. 
Then as the sequence of cubes $W$'s tends to $\R^d$, 
\[
	\frac{f(G(\cP)|_W) - \Ex[f(G(\cP)|_W)]}{\sqrt{|W|}} \dto \Normal(0, \sigma^2), \quad \frac{\Var[f(G(\cP)|_W)]}{|W|} \to \sigma^2. 
\]
Here `$\dto$' denotes the convergence in distribution, $|W|$ is the volume of the cube $W$ and $\Normal(0, \sigma^2)$ denotes the normal distribution with mean zero and variance $\sigma^2$. Moreover, the limiting variance $\sigma^2$ is positive, if the limit add-one cost is non-trivial, that is, $\Prob(\Delta \neq 0) > 0$.
\end{theorem}

An extended version of Theorem~\ref{theo:1} is stated as Theorem \ref{thm:homo-mark}. It is worth mentioning that this is a result in the thermodynamic regime where the connection function $\varphi$ is fixed (and will be assumed to satisfy the condition 
$
	\int_{\R^d} \varphi(x) dx \in (0, \infty)$).
The terminology is based on the study of random geometric graphs \cite{Penrose-book} in which three main regimes are divided according to the limit of the radius $r = r(W)$: sparse regime ($r(W) \to 0$), critical or thermodynamic regime ($r(W) \to const$) and dense regime ($r(W) \to \infty$).

For the proof, we use the idea of generating the random connection model $(\lambda, \varphi)$ from a marked Poisson point process in \cite{Last-Nestmann-Schulte-2018}. We actually establish general CLTs for weakly stabilizing functionals on  marked Poisson point processes from which the above theorem is just a particular case. Examples of weakly stabilizing functionals include isomorphic subgraph counts, (isomorphic) component counts,  Betti numbers of the clique complex of a graph, and the size of the biggest component. Thus, CLTs for those quantities are obtained from the above general result. Note that by the approach in \cite{Last-Nestmann-Schulte-2018}, CLTs with rate of convergence for isomorphic component counts were established. By approximation, a CLT (without rate) for the total number of connected components was then derived. CLTs for Betti numbers in this paper are generalizations of that result, because the zeroth Betti number is nothing but the number of connected components. Moreover, in Section 4, we establish the CLT for the size of the biggest cluster of the  random graph provided that $\lambda$ is large enough and $\varphi$ satisfies two conditions (C1) and (C2). Roughly speaking, the condition (C1) requires that $\varphi$ is a radial function with $\lim_{x\rightarrow 0} \varphi(x)=1$, while  (C2) is a moment condition on $\varphi$. As far as we are concerned, our results for Betti numbers and the size of the biggest component of random connection models  have not been known before.

The paper is organized as follows. In Section 2, we establish a general result for weakly stabilizing functionals on marked Poisson point processes, and then the CLT  for random connection models with general functionals is derived. Thenceforth,  we apply these results to establish CLTs for isomorphic subgraph counts and Betti numbers in Section 3, and for the size of the biggest component in Section 4.

\section{General results}
\subsection{CLT for weakly stabilizing functionals on marked Poisson point processes}
In this sub-section, we consider a random graph with marks built on a special marked Poisson point process under which the random connection model $(\lambda, \varphi)$ can be generated.  
Let $\hat \eta$ be a Poisson point process on $\bS:= \R^d \times [0,1] \times [0,1]^{\N \times \N}$ with the intensity measure $\lambda \ell_d \otimes \ell \otimes \Q$, where $\lambda>0$ is a constant, $\ell_d$ is the Lebesgue measure on $\R^d$, $\ell$ is the Lebesgue measure on $[0,1]$ and $\Q = \ell^{\otimes\N \times \N}$ is the product measure of $\ell$ on $\bM := [0,1]^{\N \times \N}$. To a point $(x, t, M) = (x,t,(u_{i,j})) \in \bS$, the first component points out the location in $\R^d$, the second one is regarded as its birth time and the third one is a double sequence of marks.

Let $\{B_k\}_{k\in \N}$ be an enumeration of all unit cubes from the lattice $\Z^d$. To be more precise, each $B_k$ is of the form $\prod_{i=1}^d [n_i, n_i +1), n_i \in \Z$. For a locally finite set $\eta \subset \bS$ (the number of points of $\eta$ in any compact set is finite) whose birth times are all different, the edge marking mapping $T$ associated with $\{B_k\}$ is constructed as follows \cite{Last-Nestmann-Schulte-2018}. We first order the points of $\eta$ in each cube $B_k$ according to their birth times. Then for two points $s_1=(x,t,(u_{i,j}))$ and $s_2=(y, s, (v_{i,j}))$ in $\eta$ with $t < s$ and $s_1$ the $m$th oldest point in $B_n$, the edge $\{x, y\}$ is marked with $v_{m,n}$. The resulting image $T(\eta)$ consisting of points of the form $(\{x, y\}, u)$ is viewed as a graph with marks on the set of the first components of $\eta$. Formally, $T$ is a measurable map from $\bN(\bS)$ to $\bN((\R^d)^{[2]} \times [0,1])$, where $\bN(\bX)$ denotes the space of locally finite subsets of a topological space $\bX$ and $(\R^d)^{[2]}$ denotes the space of undirected edges. Given a connection function $\varphi$, which is a measurable symmetric function $\varphi \colon \R^d \to [0,1]$, a random connection model with parameters $(\lambda, \varphi)$ can be generated from $\hat \eta$ by 
\begin{equation}\label{RCM}
	\hat \eta \mapsto T(\hat \eta) \mapsto \Big\{\{x, y\} : (\{x, y\}, u) \in T(\hat \eta),  u < \varphi(x - y) \Big\}.
\end{equation}
We will study more about RCMs in the next section.

For $\eta \in \bN(\bS)$, and $W\subset \R^d$, denote by $\eta |_W$ the restriction of $\eta$ on $\{(x, t, M) : x \in W\}$ and $T(\eta)|_W$ the induced subgraph of $T(\eta)$ with vertices in $W$. Note that by the construction $T(\hat \eta)|_W = T(\hat \eta |_W)$, if $W$ is a union of sets from the collection $\{B_k\}$. For general $W$, two graphs  $T(\hat \eta)|_W$ and $T(\hat \eta |_W) $ have the same set of vertices, but edges may be different. However, they have the same distribution. This is because conditional on the configuration of points in $W$, each edge is independently marked with a random variable uniformly distributed on $[0,1]$.

 Let $f$ be a (measurable) functional defined on finite subsets of $(\R^d)^{[2]} \times [0,1]$. Then the add-one cost of $f$, the functional on finite subsets of $\bS$, is defined as
\begin{equation}\label{add-one}
	D_{(x,t,M)}(\eta) = f(T(\eta \cup \{(x,t,M)\})) - f(T(\eta \cup \{(x,t,M)\}) \setminus \{x\} ),  \quad \eta \subset \bS.
\end{equation}
Here for $(x,t,M) \in \eta$, the graph $T(\eta) \setminus \{x\}$ is obtained from $T(\eta)$ by removing the vertex $x$ and all corresponding edge marks.

Set $\hat \Omega = \Omega \times \bM$ and $\hat \Prob = \Prob \otimes \Q$, where $(\Omega, \cF, \Prob)$ is the underlying probability space for the point process $\hat \eta$. We will use $\hat \Ex$ to denote the expectation with respect to $\hat \Prob$.

\begin{definition}\label{defn:weakly-stabilizing}
\begin{itemize}
\item[(i)] The functional $f$ is said to be translation invariant if for any $z \in \R^d$, and any finite set $\{(x_i, y_i, u_i)\}_{i \in I} \subset (\R^d)^{[2]} \times [0,1]$, 
 \[
 	f(\{(x_i, y_i, u_i)\}_{i \in I}) = f(\{(z + x_i, z+ y_i, u_i)\}_{i \in I}).
 \]
 Here for simplicity, $(x, y, u)$ denotes an element in $(\R^d)^{[2]}\times [0,1]$.
 
 \item[(ii)] The functional $f$ is said to be weakly stabilizing if it is translation invariant and there is a random variable $\Delta_1 = \Delta_1(\omega, M)$ (defined on $\hat \Omega$) such that for any sequence of cubes $\{W_n\}_n$ tending to $\R^d$, 
	\[
		D_{(\zero,1,M)} (\hat \eta |_{W_n}) \hatPto \Delta_1.
	\]
Here `$\hatPto$' denotes the convergence in probability with respect to $\hat \Prob$ and $\zero = (0, \dots, 0)\in \R^d$ denotes the origin.
\end{itemize}
\end{definition}

\begin{remark}
Note that $T(\hat \eta |_{W} \cup \{\zero,1,M\})$ is obtained from $T(\hat \eta |_{W})$ by adding the vertex $\zero$ and new edges connected to it. Thus, 
\[
	D_{(\zero,1,M)} (\hat \eta |_{W}) = f(T(\hat \eta |_{W} \cup \{\zero,1,M\})) - f(T(\hat \eta |_{W})).
\]
\end{remark}

From now on, assume that the functional $f$ is translation invariant.
The following criterion might be useful to check the weak stabilization.
\begin{proposition}\label{prop:increasing-sequence}
	Assume that for any increasing sequence of cubes $\cW = \{W_n\}_{n = 1}^\infty$ tending to $\R^d$, the sequence $\{D_{(\zero,1,M)}(\hat \eta|_{W_n})\}$ converges in probability to a limit $\Delta^{(\cW)}$. Then the functional $f$ is weakly stabilizing. 
\end{proposition}

\begin{proof}
	We first show that the limit $\Delta^{(\cW)}$ is unique. Let  $\cV=\{V_n\}$ and $\cW=\{W_n\}$ be two increasing sequences of cubes tending to $\R^d$. Then by the assumption, 
\[
	D_{(\zero,1,M)} (\hat \eta |_{W_n}) \hatPto \Delta^{(\cW)}, \quad 	D_{(\zero,1,M)} (\hat \eta |_{V_n}) \hatPto \Delta^{(\cV)} .
\]
We form a new sequence from subsequences of $\cV$ and $\cW$ in a way that 
		\[
		V_1 \subset W_{i_1} \subset V_{j_1} \subset W_{i_2} \subset \cdots \nearrow \R^d.
	\]
Along this sequence, the limit of the add-one cost exists, implying that $\Delta^{(\cW)} = \Delta^{(\cV)}$ ($\hat \Prob$-almost surely). Thus, there is a random variable $\Delta$ such that for any increasing sequence $\{W_n\}$ tending to $\R^d$, 
\[
		D_{(\zero,1,M)} (\hat \eta |_{W_n}) \hatPto \Delta.
\]

Now let $\{W_n\}$ be an arbitrary sequence of cubes tending to $\R^d$. Assume for contradiction that $\{D_{(\zero,1,M)} (\hat \eta |_{W_n})\}$ does not converge in probability to $\Delta$. Then there are $\varepsilon > 0, \delta > 0$, and a subsequence $\{W_{n_k}\}$ such that 
\[	
	\hat\Prob (| 	D_{(\zero,1,M)} (\hat \eta |_{W_{n_k}}) - \Delta| \ge \varepsilon) > \delta.
\]
Since the sequence $\{W_{n_k}\}$ tends to $\R^d$, we can always extract a further increasing subsequence along which  the sequence of the add-one cost converges to $\Delta$, making a contradiction. The proof is complete. 
\end{proof}

\begin{definition}
The functional $f$ is said to satisfy a moment condition if for some $p > 2$,
\begin{equation}\label{moment-condition}
	\sup_{\zero \in W: \text{cube}}  \hat \Ex [|D_{(\zero, 1, M)} (\hat \eta |_{W})| ^ p]< \infty.
\end{equation}
\end{definition}

Assume that the functional $f$ is weakly stabilizing and satisfies the above moment condition. Then for any sequence of cubes $\{W_n\}$ tending to $\R^d$, 
\begin{equation}\label{Lq-convergence}
	D_{(\zero, 1, M)}(\hat \eta |_{W_n}) \to \Delta_1 \quad \text{in $L^{q}(\hat \Omega)$},
\end{equation}
for $1\le q<p$.
This is a consequence of a fundamental result in probability theory (the corollary following Theorem 25.12 in \cite{Billingsley}).

Our main result in this paper is the following central limit theorem.
\begin{theorem}\label{thm:main}
	Assume that the functional $f$ is weakly stabilizing and satisfies the moment condition~\eqref{moment-condition}. Then for any sequence of cubes $W$'s tending to $\R^d$, 
	\[
	\frac{f(T(\hat \eta|_{W})) - \Ex[f(T(\hat \eta|_{W}))]}{\sqrt{|W|}} \dto \Normal(0, \sigma^2),
	\]
for a constant $\sigma^2 \ge 0$ given in~\eqref{limiting-variance} below. Here recall that $|W|$ denotes the volume, or the Lebesgue measure $\ell_d(W)$ of $W$. Moreover, the limiting variance $\sigma^2$ is positive, if the limit add-one cost $\Delta_1$ is non-trivial, that is, $\hat \Prob(\Delta_1 \neq 0) > 0$.
\end{theorem}

\begin{remark}

\begin{itemize}
\item[(i)]
For marked Poisson point processes, central limit theorems have been established for functionals $h$ of the form 
\[
	h(\eta) = \sum_{i \in I} \xi((x_i, t_i, M_i), \eta )),
\]
defined on finite subset $\eta = \{(x_i, t_i, M_i)\}_{i \in I} $ of $\bS$, provided that the functional $\xi$ is stabilizing plus some moment conditions \cite{BY-2005, PY-2005}. Isomorphic subgraph counts are typical examples of such functionals in which $\xi((x_i, t_i, M_i), \eta ))$ is the number of isomorphic subgraphs containing the point $x_i$, divided by a constant. For RCMs studied in the next section, we will show that subgraph counts are weakly stabilizing. However, when the connection function satisfies $\varphi \in (0,1)$, one may immediately see that in the simplest case where the number of edges is considered, the corresponding functional $\xi$ is not stabilizing (in the sense of \cite[\S 2.3.1]{BY-2005} or \cite[Definition 2.1]{PY-2005}). And thus, those general results are not applicable to the RCM~\eqref{RCM} constructed from a marked Poisson point process.

\item[(ii)] A CLT for weakly stabilizing functionals (in the case without marks) was first established in \cite{PY-2001} under a fourth moment condition (a similar condition as the moment condition \eqref{moment-condition} with $p = 4$). It was slightly improved to the case $p > 2$ in \cite{Trinh-2019}. We extend the approach in \cite{Trinh-2019} to prove Theorem~\ref{thm:main}. It is worth noting that a direct generalization of CLTs from the above two papers to a marked  case would lead to a CLT for a functional $h$ defined on finite subset $\eta = \{(x_i, t_i, M_i)\}_{i \in I} $ of $\bS$ with the stabilization concept being defined by using the add-one cost 
\[
	D_{(x, t, M)}h(\eta) = h(\eta \cup \{(x, t, M)\}) - h(\eta).
\]
To apply to the RCM~\eqref{RCM}, we would consider a functional of the form $h= f \circ T$, and thus the add-one cost is given by
\[
	D_{(x, t, M)}h(\eta)  = f(T(\eta \cup \{(x, t, M)\})) - f(T(\eta)).
\]
From the construction of $T$, it is clear that $T(\eta)$ is not a subgraph of $T(\eta \cup \{(x, t, M)\})$, in general, which causes a difficulty in this direction. The idea here is to use an add-one cost defined in the equation~\eqref{add-one}, the difference of $f$ on $T(\eta \cup \{(x, t, M)\})$ and its subgraph $T(\eta \cup \{(x, t, M)\}) \setminus \{x\}$, which is originated from \cite{Last-Nestmann-Schulte-2018}  to define the weak stabilization.
The main contribution of this paper is to introduce a suitable generalized concept of weak stabilization and to establish the limiting variance formula stated in Lemma~\ref{lem:limit-variance}. 
\end{itemize}
\end{remark}

\begin{remark}[A quenched CLT] We state here  a quenched version of Theorem \ref{thm:main}.  For simplicity, assume that the underlying probability space is written as the product 
\[
	(\Omega, \cF, \Prob) = (\Omega_1, \cF_1, \Prob_1) \times (\Omega_2, \cF_2, \Prob_2)
\] for which the first component of $\hat \eta$ is defined on $\Omega_1$, and the second and the third ones are defined on $\Omega_2$, that is,
\[
	\hat \eta(\omega) = \{(x(\omega_1), t(\omega_2), M(\omega_2))\}.
\]
Assume that the functional $f$ satisfies the conditions in Theorem~\ref{thm:main}.
Let 
\[
Z_n(\omega_1, \omega_2)=\frac{f(T(\hat \eta|_{W_n})) - \Ex_2[f(T(\hat \eta|_{W_n}))]}{\sqrt{n}},
\]
where $\Ex_2$ denotes the expectation with respect to $\Prob_2$.
Then there exists $0\leq \sigma_q^2 \leq \sigma^2$ (the variance in Theorem \ref{thm:main}), such that with high probability (in $\omega_1$) the random variables $(Z_n(\omega_1, \cdot))_{n\geq 1}$ converge weakly to $\Normal(0,\sigma_q^2)$. The detailed statement (see \eqref{qclt-zn}) and its proof are given in the Appendix A.
\end{remark}

We need some preparations before proving the main result. A cube is called a lattice cube if it is of the form $\prod_{i=1}^d [n_i, n_i + m)$, with $n_i \in \Z$ and $m \in \N$. Clearly, a lattice cube is a union of cubes from the collection $\{B_k\}$.

\begin{lemma}\label{lem:add-t}
Assume that the functional $f$ is weakly stabilizing. Then there is a random variable $\Delta_t$ (defined on $\hat \Omega$) such that for any sequence of cubes $\{W_n\}$ tending to $\R^d$, 
	\[
		D_{(\zero,t,M)} (\hat \eta |_{W_n}) \hatPto \Delta_t.
	\]

\end{lemma}
\begin{proof}
The proof is based on the following two observations 
\begin{itemize}
	\item[\rm (i)] under $\hat \Prob$, the two graphs $T(\hat \eta|_W \cup \{(\zero, t, M)\})$ and $T(\hat \eta|_W \cup \{(\zero, 1, M)\})$ have the same distribution;
	\item[\rm (ii)] for any lattice cube $V$ with $\zero \in V \subset W$, 
	\[
		D_{(\zero,t,M)} (\hat \eta |_W) - D_{(\zero,t,M)} (\hat \eta |_V) \overset{d}{=} D_{(\zero,1,M)} (\hat \eta |_W) - D_{(\zero,1,M)} (\hat \eta |_V). 
	\] 
\end{itemize}
Here `$\overset{d}{=}$' denotes the equality in distribution.
Then similar arguments as those will be used in the proof of Proposition~\ref{prop:D-tilde} work to show the weak stabilization of $D_{(\zero, t, M)}$ for any $t \in [0,1]$. Let us omit the details to continue the main stream. 
\end{proof}

For $\eta \in \bN(\R^d \times [0,1] \times \bM)$ and $t \in [0,1]$, we write $\eta_t$ for the restriction of $\eta$ to $\R^d \times [0,t) \times \bM$, and $\Ex[\cdot | \hat \eta_t]$ denotes the conditional expectation with respect to the sigma-field generated by $\hat \eta_t$.

\begin{lemma}\label{lem:limit-variance}
Assume that the functional $f$ is weakly stabilizing and satisfies the moment condition~\eqref{moment-condition}. Then for any sequence of cubes $W$'s tending to $\R^d$, 
\begin{equation}\label{limiting-variance}
	 \frac{\Var[f(T(\hat \eta|_{W}))]}{|W|}\to \lambda  \int_0^1 \hat \Ex[ \Ex[\Delta_t | \hat \eta_t]^2 ]   dt =: \sigma^2.
\end{equation}
The limiting variance $\sigma^2$ is positive, if $\hat \Prob(\Delta_1 \neq 0) > 0$. 
\end{lemma}

\begin{remark}
In the case without marks as in \cite{PY-2001, Trinh-2019}, the limit $\Delta = \Delta_1$ does not depend on $t$ and $M$, and thus the limiting variance is written as
\[
	\sigma^2 = \lambda \int_0^1 \Ex[\Ex[\Delta | \hat \eta_t]^2] dt.
\]
The above lemma shows that $\sigma^2 > 0$, if $\Prob(\Delta \neq 0) > 0$. Note that under the assumption of strong stabilization, Theorem 2.1 in \cite{PY-2001} states that the limiting variance $\sigma^2$ is positive, if $\Delta$ is nondegenerate, that is, $\Delta$ is not a constant. 
\end{remark} 

The proof of the above lemma relies on the following variance formula.

\begin{lemma}[{\cite[Theorem~5.1]{Last-Nestmann-Schulte-2018}}]
Let $f \colon \bN((\R^d)^{[2]} \times [0,1]) \to \R$ be measurable with $\Ex[f(T(\hat \eta|_W))^2] < \infty$, where $W \subset \R^d$. Then 
\[
	\Var[f(T(\hat \eta|_W))] = \lambda \int_{W} \int_0^1 \hat\Ex[\Ex[ D_{(x,t,M)}(\hat \eta|_{W}) | \hat \eta_t]^2] dt dx.
\]
\end{lemma}

\begin{proof}[Proof of Lemma~{\rm\ref{lem:limit-variance}}]
Similar to the convergence \eqref{Lq-convergence}, the weak stabilization and the moment condition imply that 
	\[
		\hat \Ex [| D_{(\zero, t, M)} (\hat \eta |_{W_n}) - \Delta_t|^{2}] \to 0,  
	\]
along any sequence of cubes $\{W_n\}$ tending to $\R^d$. It follows that 
	\[
		 \int_0^1 \hat \Ex \left[ |D_{(\zero, t, M)} (\hat \eta|_{W_n}) - \Delta_t|^2 \right] dt  \to 0.
	\]
Then, by using Jensen's inequality for conditional expectation, we obtain that 
\[
	\int_0^1 \hat \Ex \left[ | \Ex[ D_{(\zero, t, M)} (\hat \eta|_{W_n}) | \hat \eta_t ] - \Ex[ \Delta_t | \hat \eta_t ] |^2 \right] dt  \to 0.
\]	
Consequently, 
\[
	h(W_n) : =  \int_0^1 \hat \Ex \left[ \Ex[ D_{(\zero, t, M)} (\hat \eta|_{W_n}) | \hat \eta_t ]^2 \right] dt  \to  \int_0^1 \hat \Ex \left[ \Ex[ \Delta_t | \hat \eta_t ]^2 \right] dt =:a^2.
\]
The convergence holds for any sequence of cubes $\{W_n\}$ tending to $\R^d$. Thus, it is straightforward to show that for any $\varepsilon > 0$, there is a radius $r > 0$ such that 
\begin{equation}\label{epsilon}
	| h(V) - a^2 | < \varepsilon, \quad \text{if  $B_r(\zero) \subset V$.}
\end{equation}
Here $B_r(\zero) = \{x \in \R^d: |x| \le r\}$ denotes the closed ball centered at $\zero$ of radius $r$.

It now follows from the variance formula and the translation invariance that
\begin{align*}
	\Var[f(T(\hat \eta|_{W}))] &= \lambda \int_W  \int_0^1 \hat \Ex[\Ex[D_{(x,t,M)} (\hat \eta|_{W}) | \hat \eta_t]^2] dt  dx \\
	&= \lambda \int_W h(W-x) dx.
\end{align*}
For given $\varepsilon > 0$, take $r$ such that the condition \eqref{epsilon} holds. Then divide the above integral into two parts according to $B_r(\zero) \subset W-x$ or not. For the part with $B_r(\zero) \subset W-x$, the integrand $h(W-x)$ is different from $a^2$ by at most $\varepsilon$, while the integral over the other part divided by $|W|$ clearly vanishes as $W$ tends to $\R^d$. Consequently, 
\[
    	\frac{\Var[f(T(\hat \eta|_{W}))]}{|W|} = \frac{\lambda}{|W|} \int_W h(W-x) dx  \to \lambda a^2=\sigma^2 \quad \text{as} \quad W \to \R^d,
\]
which proves the desired convergence \eqref{limiting-variance}.

Next, we show the positivity of $\sigma^2$ under the condition that $\hat \Prob(\Delta_1 \neq 0) > 0$. Our aim is to show the continuity of $\hat \Ex[\Ex[\Delta_t |\hat \eta_t]^2]$ at $t=1$, that is, 
\begin{equation}\label{continuity-at-1}
	\hat \Ex[\Ex[\Delta_t |\hat \eta_t]^2]  \to \hat \Ex[\Delta_1^2] \quad \text{as} \quad t \to 1.
\end{equation}
This clearly implies the positivity of $\sigma^2$, because $\Ex[\Delta_1^2] >0$. We will show the continuity through several steps.

{\bf Step 1.}
Recall that as $W$ tends to $\R^d$,  
\[
	\hat \Ex [| D_{(\zero, t, M)} (\hat \eta |_W) - \Delta_t|^{2}] \to 0,
\]
and moreover the expectation does not depend on $t$, if $W$ is a lattice cube.

{\bf Step 2.}
For any finite cube $W$,
\[
	D_{(\zero, t, M)}(\hat \eta|_{W}) \to D_{(\zero,1,M)}(\hat \eta|_{W}) \quad \text{in probability as} \quad t \to 1.
\]
This is because the two functionals coincide on the event that there is no point in $W \times [t,1] \times \bM$ whose probability tends to $1$ as $t \to 1$. Then the convergence in $L^{2}$ holds as a consequence of the moment condition.

{\bf Step 3.} The results in Step 1 and Step 2, together with the triangular inequality, imply that as $t \to 1$,
\[
	\Delta_t  \to \Delta_1 \quad \text{in $L^{2}$.}
\]
Then using Jensen's inequality for conditional expectation, we obtain that
\[
	\hat \Ex[\Ex[\Delta_t - \Delta_1 | \hat \eta_t]^2] \to 0 \quad \text{as} \quad t \to 1,
\]
and thus,
\[
	\hat \Ex[\Ex[\Delta _t | \hat \eta_t]^2] - \hat \Ex [\Ex[\Delta_1 | \hat \eta_t]^2]   \to 0 \quad \text{as} \quad t \to 1.
\]

{\bf Step 4.} We claim that for any finite cube $W$,
\[
	\Ex[D_{(\zero,1,M)} (\hat \eta |_{W}) | \hat \eta_t] \to D_{(\zero,1,M)}(\hat \eta|_W) \quad \text{in probability, and then in $L^{2}$ as }t \to 1.
\]
It suffices to show that for each $M$, the above convergence holds in probability with respect to $\Prob$. Let $M$ be fixed. First we write the conditional expectation as 
\begin{align*}
	\Ex[D_{(\zero,1,M)} (\hat \eta |_{W}) | \hat \eta_t] = \Ex^{\hat \eta^t} [D_{(\zero, 1, M)} (\hat \eta_t |_W + \hat \eta^t|_W)],
\end{align*}
where $\Ex^{\hat \eta^t}$ denotes the expectation  with respect to a Poisson point process $\hat \eta^t$ on $\R^d \times [t,1] \times \bM$ independent of $\hat \eta_t$. Then by expressing the conditional expectation  further as 
\[
	D_{(\zero, 1, M)} (\hat \eta_t |_W) \Prob(A_t) +  \Ex^{\hat \eta^t} [D_{(\zero, 1, M)} (\hat \eta_t |_W + \hat \eta^t|_W) \one_{A_t^c}],
\]
where $A_t$ is the event that $\hat \eta^t$ has no point in $W\times [t,1]\times \bM$, we see that
\begin{align*}
	&\Ex[D_{(\zero,1,M)} (\hat \eta |_{W}) | \hat \eta_t] - D_{(\zero, 1, M)} (\hat \eta_t |_W) \\
	&=D_{(\zero, 1, M)} (\hat \eta_t |_W) (\Prob(A_t) -1) +  \Ex^{\hat \eta^t} [D_{(\zero, 1, M)} (\hat \eta_t |_W + \hat \eta^t|_W) \one_{A_t^c}] \\
	&\to 0 \quad \text{in probability as } t \to 1.
\end{align*}
Here H\"older's inequality has been used to show the second term in the second last equation converges to zero.
In addition, $D_{(\zero,1,M)} (\hat \eta_t |_{W}) \to D_{(\zero,1,M)}(\hat \eta|_W)$ in probability (by the same reason as in Step 2). These imply the desired convergence.

{\bf Step 5.} Take the limit as $W\to \R^d$ in Step 4, we obtain 
\[
	\Ex[\Delta_1 | \hat \eta_t] \to \Delta_1 \quad \text{in $L^{2}$ as } t \to 1,
\]
which, together with Step 3, yields the continuity~\eqref{continuity-at-1}. The proof is complete.
\end{proof}

We also need the following Poincar\'e inequality which is a direct consequence of the variance formula by using Jensen's inequality.

\begin{lemma}[\cite{Last-Nestmann-Schulte-2018}]
Let $f$ be a functional defined on finite subsets of $(\R^d)^{[2]}\times [0,1]$. Assume that 
\[
	\Ex[f(T(\hat \eta|_W))^2] < \infty.
\]
Then the following Poincar\'e inequality holds
\[
	\Var[f(T(\hat \eta|_W))] \le \lambda \int_W \hat \Ex[D_{(x,1,M)}(\hat \eta|_W)^2] dx.
\]

\end{lemma}

\begin{proof}[Proof of Theorem~{\rm\ref{thm:main}}]
We sketch some key steps in the argument because it is similar to the proof of Theorem~3.1 in \cite{Trinh-2019}. Let us consider the sequence of cubes $W_n := [-n^{1/d}/2, n^{1/d}/2)^d$, where $n$ needs not be an integer number.
Because of the translation invariance, it suffices to show that as $n \to \infty$,
\begin{equation}\label{aim}
	\frac{f(T(\hat \eta|_{W_n})) - \Ex[f(T(\hat \eta|_{W_n}))]}{\sqrt{n}} \dto \Normal(0, \sigma^2), \quad \frac{\Var[f(T(\hat \eta|_{W_n}))]}{n} \to \sigma^2,
\end{equation}
for some $\sigma^2 \ge 0$. For $L>0$ with $L^{1/d}$ an integer number and for each $n$, divide the cube $W_n$ according to the lattice $L^{1/d} \Z^d$ and let $\{C_i\}_{i = 1}^{\ell_n}$ be the lattice cubes entirely contained in $W_n$. Then it follows from the construction of the graph $T$ that 
\begin{align*}
	X_{n,L} &:= \frac{1}{\sqrt n} \sum_{i = 1}^{\ell_n} \Big(  f(T(\hat \eta|_{C_i})) - \Ex[f(T(\hat \eta|_{C_i}))]	\Big) \\
	& = \frac{1}{\sqrt n} \sum_{i = 1}^{\ell_n} \Big(  f(T(\hat \eta|_{W_n})|_{C_i} ) - \Ex[f(T(\hat \eta|_{W_n})|_{C_i})]	\Big).
\end{align*}

The first expression shows that $X_{n,L}$ is a sum of i.i.d.~(independent identically distributed) random variables. Thus, for fixed $L$, a central limit theorem for $\{X_{n, L}\}$ holds, that is, 
\begin{equation}\label{CLT-XnL}
	X_{n, L} \dto \Normal(0, \sigma_L^2), \quad \Var[X_{n, L}] \to \sigma_L^2 = L^{-1} \Var[f(T(\hat \eta |_{C_i}))].
\end{equation}

The second expression helps us to make use of the Poincar\'e inequality
\begin{align*}
	\Var &
	\left [
		\frac{f(T(\hat \eta|_{W_n})) - \Ex[f(T(\hat \eta|_{W_n}))]}{\sqrt{n}} - X_{n, L}  
	\right ] \\
	&\le \frac{\lambda}{n} \int_{W_n}  \hat \Ex\bigg[ \bigg|D_{(x, 1, M)}(\hat \eta|_{W_n})  - \sum_{i=1}^{\ell_n} D_{(x,1,M)}(\hat \eta|_{C_i})\I{( x \in C_i}) \bigg|^2 \bigg] dx\\
	&=\frac{\lambda}{n} \int_{W_n \setminus \cup_i C_i}  \hat \Ex  [  |D_{(x, 1, M)}(\hat \eta|_{W_n})|^2 ] dx \\
	&\quad + \sum_{i = 1}^{\ell_n} \int_{C_i}  \hat \Ex  [  |D_{(x, 1, M)}(\hat \eta|_{W_n})  -   D_{(x,1,M)}(\hat \eta|_{C_i})    |^2  ] dx.
\end{align*}
Here $\I$ denotes the indicator function.
Then using the weak stabilization together with the moment condition, we can argue in exactly the same way as in  the proof of Theorem~3.1 in \cite{Trinh-2019} to show that 
\begin{equation}\label{CLT-approximation}
	\lim_{L \to \infty} \limsup_{n \to \infty} \Var
	\left [
		\frac{f(T(\hat \eta|_{W_n})) - \Ex[f(T(\hat \eta|_{W_n}))]}{\sqrt{n}} - X_{n, L}  
	\right ] = 0.
\end{equation}
The two equations \eqref{CLT-XnL} and \eqref{CLT-approximation} imply our desired CLT~\eqref{aim} (see \cite[Lemma 2.2]{Trinh-2019}). The proof is complete.
\end{proof}

\begin{corollary}\label{cor:multi}
	Assume that functionals $\{f_i\}_{i = 1}^m$ are weakly stabilizing and satisfy the moment condition.
	Then as the sequence of cubes $W$'s tends to $\R^d$,
	\[
		\left( \frac{f_i(T(\hat \eta|_{W})) - \Ex[f_i(T(\hat \eta|_{W}))]}{\sqrt{|W|}}\right)_{i=1}^m \dto \Normal(0,  \Sigma),
	\]
where $\Sigma= (\sigma_{ij})_{i,j=1}^m$ is a nonnegative definite matrix,
\begin{equation}\label{sigma-ij}
	\sigma_{ij} = \lim_{n \to \infty} \frac{\Cov [f_i(T(\hat \eta|_{W})), f_j(T(\hat \eta|_{W}))]}{|W|}.
\end{equation}
Here $ \Normal(0,  \Sigma)$ denotes the multidimensional Gaussian distribution with mean zero and covariance matrix $\Sigma$.
\end{corollary}
\begin{proof}
Observe that the desired multidimensional CLT follows, if we can show that for any ${\mathbf a} =(a_1, \dots, a_m) \in \R^m$, the following hold for 
	$
		f = \sum_{i = 1}^m a_i f_i,
	$ 
\[
\frac{f(T(\hat \eta|_{W})) - \Ex[f(T(\hat \eta|_{W}))]}{\sqrt{|W|}} \dto \Normal(0,  \sigma^2_f), \quad  \frac{\Var[f(T(\hat \eta|_{W}))]}{|W|} \to   \sigma_f^2 = {\mathbf a}^t \Sigma {\mathbf a}.
\]

The functional $f$ is clearly weakly stabilizing and satisfies the moment condition, and hence a CLT for $f$ follows from Theorem~\ref{thm:main}. To see the convergence of the covariance and the formula $\sigma_f^2 = {\mathbf a}^t \Sigma {\mathbf a}$, it remains to show the convergence~\eqref{sigma-ij}. However, it is an easy consequence of the convergence of variances when applying Theorem~\ref{thm:main} to the functionals $(f_i \pm f_j)$ by noting that $\Cov [X, Y] = \frac{1}{4} (\Var[X+Y]  - \Var[X-Y])$. The proof is complete.
\end{proof}

We conclude this sub-section by discussing further equivalent conditions for the weak stabilization. Consider the add one-cost functional in a slightly different way
\[
	\tilde D_{(\zero,1,M)} (W) = f(T(\hat \eta \cup \{(\zero,1,M)\} ) |_{W}) - f(T(\hat \eta) |_{W}) .
\]
Here we first construct the infinite graph $T(\hat \eta \cup \{(\zero,1,M)\} )$ and then take the restriction. Its advantage is the increasing property of a sequence of graphs.
The two add-one cost functionals coincide, if $W$ is a union of cubes from the collection $\{B_k\}$.
\begin{proposition}\label{prop:D-tilde}
	The following are equivalent 
	\begin{itemize}
	 \item[\rm (i)] the functional $f$ is weakly stabilizing; 
	 \item[\rm (ii)] for any sequence of cubes $\{W_n\}$ tending to $\R^d$,
	 \[
	 	\tilde D_{(\zero,1,M)} (W_n) \hatPto \Delta_1; 
	 \]
	 \item[\rm (iii)] the sequence $\{\tilde D_{(\zero,1,M)} (W_n)\}$ converges in probability to a limit for any sequence of increasing cubes $\{W_n\}$ tending to $\R^d$.
	 \end{itemize}
\end{proposition}

\begin{proof}
The equivalence of (ii) and (iii) is quite similar to Proposition~\ref{prop:increasing-sequence}, and hence its proof is omitted. We now prove the equivalence of (i) and (ii).

Let $V \subset W$ be a lattice cube.
Note that the two graphs  $T(\hat \eta |_W \cup \{(\zero,1,M)\})$ and  $T(\hat \eta \cup \{(\zero, 1, M)\})|_W$ have the same distribution. In addition, since $V$ is a lattice cube, $D_{(\zero,1,M)}(\hat \eta |_W) - D_{(\zero,1,M)}(\hat \eta |_V)$ can be written as a function of $T(\hat \eta \cup \{(\zero, 1, M)\})|_W$ by restriction. In the same manner, $\tilde D_{(\zero,1,M)} (W) - \tilde D_{(\zero,1,M)} (V)$ can be written as the same function of $T(\hat \eta \cup \{(\zero, 1, M)\})|_W$. Consequently, 
\begin{equation}\label{VW}
	D_{(\zero,1,M)}(\hat \eta |_W) - D_{(\zero,1,M)}(\hat \eta |_V) \overset{d}{=} \tilde D_{(\zero,1,M)} (W) - \tilde D_{(\zero,1,M)} (V).
\end{equation}

	Let us show that (i) implies (ii). Assume that (i) holds.  Let $\{W_n\}$ be any sequence of cubes tending to $\R^d$. We take a sequence of lattice cubes $\{V_n\}$ tending to $\R^d$ such that $V_n \subset W_n$, for each $n$. We assume without loss of generality that $\zero \in V_n$, for any $n$. 
The condition~(i) implies that 
\[
	D_{(\zero,1,M)}(\hat \eta |_{W_n}) \hatPto \Delta_1, \quad 	D_{(\zero,1,M)}(\hat \eta |_{V_n}) \hatPto \Delta_1.
\]
It then follows that 
\[
	D_{(\zero,1,M)}(\hat \eta |_{W_n}) - D_{(\zero,1,M)}(\hat \eta |_{V_n}) \hatPto 0,
\]
and hence 
\begin{equation}\label{tildeDW}
	\tilde D_{(\zero,1,M)} (W_n) - \tilde D_{(\zero,1,M)} (V_n) \hatPto 0,
\end{equation}
by the identity in distribution~\eqref{VW}. In addition, since $V_n$ is a lattice cube, 
	\begin{equation}\label{tildeDV}
		 \tilde D_{(\zero,1,M)} (V_n) = D_{(\zero,1,M)}(\hat \eta |_{V_n}) \hatPto \Delta_1. 
	\end{equation}
Adding the two equations \eqref{tildeDW} and \eqref{tildeDV}, we get the desired convergence in the statement of (ii). The converse can be proved similarly.
The proof is complete.
\end{proof}

\subsection{CLT for random connection models}

Let $\cP$ be a homogeneous Poisson point process on $\R^d$ with density $\lambda > 0$. Let 
	$\varphi \colon \R^d \to [0,1]$
be a measurable, symmetric function, that is, $\varphi(x) = \varphi(-x)$.
	Given a configuration $\cP$ which is a locally finite subset in $\R^d$ (almost surely), connect any two points $x, y \in \cP$ independently with probability $\varphi(x-y)$. (In general, we can consider a connectivity function $\varphi \colon \R^d \times \R^d \to [0,1]$ with $\varphi(x,y) = \varphi(y,x)$ the probability of connecting two points $x$ and $y$. The model here is the case where the translation invariance is assumed.) The resulting graph,  denoted by $G(\cP)$, is called a random connection model with parameters $(\lambda, \varphi)$. If we take $\varphi$ as
\[
	\varphi(x) = \begin{cases}
		1,	& |x| \le r,\\
		0,	& \text{otherwise},
	\end{cases} 
\]
for some $r > 0$, then the RCM  $G(\cP)$ reduces to  a random geometric graph.

	The graph $G(\cP)$ can be generated by using the random graph with marks in the previous section as follows \cite{Last-Nestmann-Schulte-2018}. Let $\hat \eta$ be a Poisson point process on $\R^d \times [0,1] \times [0,1]^{\N \times \N}$ with the intensity measure $\lambda \ell_d \otimes \ell \otimes \Q$. We regard $\cP$ as the projection of $\hat \eta$ to $\R^d$. Then define the graph $G(\cP)$ as the one with the vertex set $\cP$, and edges $\{x, y\}$, if $
		u < \varphi(x-y)$,
for $(x, y, u) \in T(\hat \eta)$.  In other words, $G(\cP) = \iota(T(\hat \eta))$ is the image of $T(\hat \eta)$ under some mapping $\iota$ defined on $\bN((\R^d)^{[2]} \times [0,1])$.

	For a bounded subset $W \subset \R^d$, let $G(\cP)|_W$ be the induced subgraph obtained from $G(\cP)$  by restricting the graph on the vertex set in $\cP|_W$. Note that $G(\cP)|_W$ has the same distribution with the graph $G(\cP|_W)$ generated by connecting a pair $x, y \in \cP|_W$ with probability $\varphi(x-y)$ independent of the others. Let $f$ be a functional defined on finite graphs. Then 
\[
	f(G(\cP)|_W) = f(\iota (T(\hat \eta) |_{W} )).
\]
Clearly, the functional $f \circ \iota$ is translation invariant. 

The functional $f$ is said to be weakly stabilizing on $G(\cP)$ if $f \circ \iota$ is weakly stabilizing as in Definition~\ref{defn:weakly-stabilizing}. In this model, this concept is equivalent to the following. Let $G(\cP \cup \{\zero\})$ be a random graph obtained from $G(\cP)$ by adding the vertex $\{\zero\}$ and new edges $(\zero, x), x \in \cP$ independently with probability $\varphi(x)$. For a bounded subset $W \subset \R^d$, let  
	\[
		D_\zero f (W) = f( G(\cP \cup \{\zero\})|_{W})  - f( G(\cP)|_{W})
	\]
be the add-one cost of $f$. Then using equivalent conditions in Proposition~\ref{prop:D-tilde}, the functional $f$ is weakly stabilizing on $G(\cP)$, if and only if one of the following two conditions holds
\begin{itemize}
	\item[(i)] there is a random variable $\Delta$ such that 
	\[
		D_\zero f (W_n) \Pto \Delta,
	\]
for any sequence of cubes $\{W_n\}_{n = 1}^\infty$ tending to $\R^d$;

\item[(ii)] for any increasing sequence of cubes $\{W_n\}$, the sequence $\{D_\zero f(W_n)\}$ converges in probability to a limit.
\end{itemize}

The following result follows directly from Theorem~\ref{thm:main} and Corollary~\ref{cor:multi}.
\begin{theorem}\label{thm:homo-mark}
\begin{itemize}
\item[\rm(i)]	Assume that a functional $f$ is weakly stabilizing on $G(\cP)$. Assume further that for some $p > 2$,
			\begin{equation}\label{p-moment}
				\sup_{\zero \in W:\text{cube}}\Ex[|D_\zero f (W)|^p] <\infty.
			\end{equation}
	Then as the sequence of cubes $W$'s tends to $\R^d$,
	\[
		\frac{f(G(\cP|_W)) - \Ex[f(G(\cP|_W))]}{\sqrt{|W|}} \dto \Normal(0,  \sigma^2), \quad  \frac{\Var[f(G(\cP|_W))]}{|W|} \to   \sigma^2.
	\]
The limiting variance is positive $(\sigma^2 >0)$, if $\Prob(\Delta \neq 0) > 0$.
\item[\rm(ii)] Assume that functionals $\{f_i\}_{i = 1}^m$ are weakly stabilizing on $G(\cP)$ and satisfy the above moment condition.	Then as the sequence of cubes $W$'s tends to $\R^d$,
	\[
		\left( \frac{f_i(G(\cP|_W)) - \Ex[f_i(G(\cP|_W))]}{\sqrt{|W|}} \right)_{i=1}^m \dto \Normal(0,  \Sigma),
	\]
where $\Sigma= (\sigma_{ij})_{i,j=1}^m$ is a nonnegative definite matrix,
\[
	\sigma_{ij} = \lim_{W \to \R^d} \frac{\Cov [f_i(G(\cP|_W)), f_j(G(\cP|_W )) ]}{|W|}.
\]
\end{itemize}
\end{theorem}

\section{Isomorphic subgraph counts and Betti numbers}
\subsection{Isomorphic subgraph counts}

Consider the random connection model $(\lambda, \varphi)$ with the assumption that 
\[
	0< m_\varphi=	\int_{\R^d} \varphi (x) dx < \infty.
\]

Let $A$ be a connected graph on $(k+1)$ vertices. For given $(k+1)$ distinct points $\{x_1, x_2, \dots, x_{k+1}\}$ in $\R^d$, denote by $\Gamma(x_1, x_2, \dots, x_{k+1})$ the random graph generated by independently drawing an edge between any two vertices $x_i, x_j$  with probability $\varphi(x_i - x_j)$. Let 
\[
	\psi_A(x_1, x_2, \dots, x_{k+1}) = \begin{cases}
	 \Prob(\Gamma(x_1, x_2, \dots, x_{k+1}) \simeq A),	&\text{if $\{x_i\}$ are distinct,}\\
	 0, &\text{otherwise},
	 \end{cases}
\]
where `$\simeq$' denotes the isomorphism of graphs.
Then it is clear that $\psi$ is translation invariant, that is, 
\[
	\psi_A(z + x_1, z+ x_2, \dots,z+ x_{k+1}) = \psi_A(x_1, x_2, \dots, x_{k+1}), \quad \text{for any $z \in \R^d$}.
\]

\begin{lemma}\label{lem:finite}
Let $A$ be a connected graph on $(k+1)$ vertices. Then the expected number of induced subgraphs containing the origin $\zero$ in $G(\cP \cup \{\zero\})$ isomorphic to $A$ is given by 
\begin{align}
	h_A &:= \frac{\lambda^k}{k!} \idotsint_{(\R^d)^k} \Prob(\Gamma(\zero, x_1, \dots, x_k) \simeq A) dx_1 \cdots dx_k \nonumber\\
	&=  \frac{\lambda^k}{k!} \idotsint_{(\R^d)^k} \psi_A(\zero, x_1, \dots, x_k) dx_1 \cdots dx_k < \infty. \label{hA}
\end{align}

\end{lemma}
\begin{proof}
We first show that the integral in \eqref{hA} is finite. Although this result was already proved in Theorem~7.1 in \cite{Last-Nestmann-Schulte-2018}, we give here a slightly different proof. We claim that for a connected graph $A$, there are at least two vertices such that after removing each of them together with all edges connected to it, the remaining graph is still connected. Indeed, let $B$ be a spanning tree of $A$, that is, a connected subgraph of $A$ with exactly $k$ edges. Then the sum of degrees of all vertices in $B$ is $2k$, implying that at least two vertices have degree one. Note that by removing a vertex of degree one from the tree, the remaining is still a tree, which proves our claim. 

Now let $A = (V, E)$ be a connected graph on $V=[k+1]:= \{1,2, \dots, k+1\}$. The graph $\Gamma(\zero, x_1, \dots, x_k)$ is isomorphic to $A$, if there is a permutation $\pi \in \mathbb S_{k+1}$ such that $\{i,j\}$ is an edge on $A$, if and only if $\{x_{\pi_i}, x_{\pi_j}\}$ is an edge on $\Gamma(x_{k+1}=\zero, x_1, \dots, x_k)$. Therefore 
\begin{align*}
&\idotsint_{(\R^d)^k} \Prob(\Gamma(\zero, x_1, \dots, x_k) \simeq A) dx_1 \cdots dx_k  \\
&\le\sum_{\pi \in \mathbb S_{k+1}} \idotsint_{(\R^d)^k} \prod_{ \{i, j\} \in E} \varphi(x_{\pi_i} - x_{\pi_j}) dx_1 \cdots dx_{k}.
\end{align*}
Then it suffices to show that 
\[
	 \idotsint_{(\R^d)^k} \prod_{ \{i, j\} \in E} \varphi(x_{i} - x_{j}) dx_1 \cdots dx_{k} < \infty.
\]
Let $m\neq k+1$ be a vertex such that the induced subgraph $A' = (V', E')$, where $V' = V \setminus \{m\}$, is still connected. Let $n$ be a vertex connected to $m$. Then 
\begin{align*}
& \idotsint_{(\R^d)^k} \prod_{ \{i, j\} \in E} \varphi(x_{i} - x_{j}) dx_1 \cdots dx_{k} \\
 &\le \idotsint_{(\R^d)^{k-1}} \bigg( \prod_{\{i, j\} \in E'} \varphi(x_i - x_j) \bigg) \bigg( \int_{\R^d} \varphi(x_n - x_m) dx_m \bigg) \prod_{l \neq m} dx_l  \\
 &= m_\varphi \times \idotsint_{(\R^d)^{k-1}}  \prod_{\{i, j\} \in E'} \varphi(x_i - x_j)  \prod_{l \neq m} dx_l  .
\end{align*}
Since $A'$ is again a connected graph, we continue this way to see that the above integral is bounded by $(m_\varphi)^k < \infty$. 

Next by the multivariate Mecke equation (Theorem~4.4 in \cite{Last-Penrose-book}), the expected number of induced subgraphs containing the origin $\zero$ in $G(\cP \cup \{\zero\})$ isomorphic to $A$ can be written as 
\begin{align*}
	&\Ex\bigg[ \sum_{ \{x_1, x_2, \dots, x_{k} \} \subset\cP}  \psi_A(\zero, x_1, x_2, \dots, x_{k})  \bigg] \\
	&= \frac{\lambda^k}{k!} \idotsint_{(\R^d)^k} \psi_A(\zero, x_1, \dots, x_k) dx_1 \cdots dx_k,
\end{align*}
which completes the proof.
\end{proof}

The graph $A$ is said to be feasible if $h_A > 0$. Equivalently, the graph $A$ is feasible, if the probability $\Prob(\Gamma(\zero, x_1, \dots, x_k) \simeq A)$ is positive on some set in $(\R^d)^k$ with positive Lebesgue measure.  In particular, in case $\varphi \in (0,1)$, any connected graph is feasible. Let $\xi_n^{(A)}$ be the number of induced subgraphs in $G(\cP_n)$ isomorphic to $A$, where $\cP_n = \cP|_{W_n}$ with $W_n = [-\frac{n^{1/d}}2, \frac {n^{1/d}}2)^d$. By direct calculation using the Mecke formula, we can show the following asymptotic behaviors, natural extensions of those for random geometric graphs in \cite[Chapter~3]{Penrose-book}.

\begin{lemma}\label{lem:covariance}
\begin{itemize}
\item[\rm(i)]	Let $A$ be a feasible connected graph on $(k+1)$ vertices. Then as $n \to \infty$,
	\[
		\frac{\Ex[\xi_n^{(A)} ] }{n}  \to  \frac{\lambda}{k+1} h_A = \frac{\lambda^{k+1}}{(k+1)!} \idotsint_{(\R^d)^k} \Prob(\Gamma(\zero, x_{[k]}) \simeq A) dx_{[k]}.
	\]
Here $x_{[k]}$ denotes the set $\{x_1, \dots, x_k\}$ and $dx_{[k]}$ stands for $dx_1 \cdots dx_k$.

\item[\rm (ii)] Let $A$ and $B$ be two feasible connected graphs on $(k+1)$ vertices and $(l+1)$ vertices with $k \le l$, respectively. Then 
\begin{align*}
	&\lim_{n \to \infty} \frac{\Cov [\xi_n^{(A)}, \xi_n^{(B)} ]}{n} \quad (=: \sigma_{A,B}) \\
	&= \sum_{m=1}^{k+1}\frac{\lambda^{k+l+2-m}}{m! (k+1-m)!(l+1-m)!} \\
	&\quad\times \idotsint_{(\R^d)^{k+l+1-m}} \Prob(\Gamma(\zero, x_{[k]}) \simeq A, \Gamma(\zero, x_{[m-1]}, y_{[l +1 - m ]} ) \simeq B) \\
	&\quad \quad \times dx_{[k]} dy_{[l+1-m]},
\end{align*}
where the two graphs $\Gamma(\zero, x_{[k]})$ and $\Gamma(\zero, x_{[m-1]}, y_{[l +1 - m ]})$ are coupling as induced subgraphs of $\Gamma(\zero, x_{[k]}, y_{[l + 1 - m]})$. 
\end{itemize}
\end{lemma}

\begin{proof}
(i) 
By the multivariate Mecke equation, we  see that
\begin{align*}
	\Ex[\xi_n^{(A)}] &= \Ex \bigg[\sum_{ x_{[k+1]} \subset\cP_n}  \psi_A(x_1, x_2, \dots, x_{k+1}) \bigg]\\
	&= \frac{\lambda^{k+1}}{(k+1)!} \idotsint_{(W_n)^{k+1}} \psi_A(x_1, x_2, \dots, x_{k+1}) dx_1 dx_2 \cdots dx_{k+1} \\
	&= \frac{\lambda^{k+1}}{(k+1)!} \int_{W_n} dx_{k+1} \idotsint_{(W_n - x_{k+1})^{k}} \psi_A(\zero, x_1, x_2, \dots, x_{k}) dx_1 dx_2 \cdots dx_{k}.
\end{align*}
Since the integral in \eqref{hA} is convergent, it follows that for any $\varepsilon > 0$, there is a radius $r>0$ such that if $B_r(\zero) \subset W$, 
\[
	\bigg| \idotsint_{W^k}  \psi_A(\zero, x_1, x_2, \dots, x_{k}) dx_1 dx_2 \cdots dx_{k} - \frac{k!}{\lambda^k}h_A \bigg| < \varepsilon.
\]
Then by dividing the integral with respect to $x_{k+1}$ into two parts according to $B_r(\zero) \subset W_n - x_{k+1}$ or not, we can deduce the desired result 
\[
	\lim_{n \to \infty } \frac{\Ex[\xi_n^{(A)}] }{n} = \frac{\lambda}{k+1} h_A.
\]

(ii) Let us begin with the following expression for  $\Ex[\xi_n^{(A)}   \xi_n^{(B)} ]$
\begin{align*}
	\Ex[\xi_n^{(A)}   \xi_n^{(B)} ] = \sum_{m = 0}^{k+1} \Ex \bigg[  \sum_{\substack{x_{[k+1]},  y_{[l+1]} \subset \cP_n,\\|x_{[k+1]} \cap  y_{[l+1]}| = m  }}  \Prob(\Gamma(x_{[k+1]}) \simeq A, \Gamma(y_{[l+1]}) \simeq B )\bigg].
\end{align*}
To be more precise, the two random graphs $\Gamma(x_{[k+1]})$ and $\Gamma(y_{[l+1]})$ are coupling as induced subgraphs of a random graph on the set $x_{[k+1]} \cup y_{[l+1]}$.
Note that the term with $m=0$ coincides with $\Ex[\xi_n^{(A)}]\Ex[ \xi_n^{(B)} ] $ (by using the multivariate Mecke equation and the fact that the two random graphs are independent). Thus, the covariance $\Cov [\xi_n^{(A)}, \xi_n^{(B)} ]$ is given by 
\[
	\Cov [\xi_n^{(A)}, \xi_n^{(B)} ] =  \sum_{m = 1}^{k+1} \Ex \bigg[  \sum_{\substack{x_{[k+1]},  y_{[k+1]} \subset \cP_n.\\|x_{[k+1]} \cap  y_{[k+1]}| = m  }}  \Prob(\Gamma(x_{[k+1]}) \simeq A, \Gamma(y_{[k+1]}) \simeq B )\bigg].
\]
 For $m \ge 1$, to choose the sets $x_{[k+1]}$ and $y_{[l+1]}$ with $m$ points in common, we first select $m$ common points, and then select the remaining points of $x$'s and $y$'s. Again, using the multivariate Mecke equation, the corresponding term can be expressed further as 
\begin{align*}
	&\frac{\lambda^{k+l+2-m}}{m! (k+1-m)!(l+1-m)!} \\
	&\quad\times \idotsint_{(W_n)^{k+l+2-m}} \Prob(\Gamma(x_{[k+1]}) \simeq A, \Gamma(x_{[m]}, y_{[l +1 - m ]} ) \simeq B) dx_{[k+1]} dy_{[l+1-m]}.
\end{align*}
Then similar to the proof of (i), we can show that the above integral, divided by $n$ (the volume of $W_n$), converges to 
\begin{align*}
	\idotsint_{(\R^d)^{k+l+1-m}} \Prob(\Gamma(\zero, x_{[k]}) \simeq A, \Gamma(\zero, x_{[m-1]}, y_{[l +1 - m ]} ) \simeq B) dx_{[k]} dy_{[l+1-m]},
\end{align*}
where the integral is finite. The proof is complete.
\end{proof}

\begin{theorem}\label{thm:CLT-subgraph}
	Let $\{A_1, \dots, A_m\}$ be feasible connected graphs. Then
	\[
		\left( \frac{ \xi_n^{(A_i)}  - \Ex[\xi_n^{(A_i)}] }{\sqrt n} \right)_{i = 1}^m  \dto \Normal(0, \Sigma), \quad \Sigma = (\sigma_{A_i, A_j})_{i,j=1}^m. 
	\]
Here $\sigma_{A_i, A_i} >0$.
\end{theorem}
\begin{remark}
Lemma~\ref{lem:covariance} implies the following weak law of large numbers  
	\[
		\frac{\xi_n^{(A)}}{n} \to \frac{\lambda}{k+1}h_A \quad \text{in probability as }  n \to \infty.
	\]
\end{remark}
\begin{proof}
Let $A$ be a feasible connected graph and let $f$ be the functional counting the number of induced subgraphs isomorphic to $A$. 
By Theorem~\ref{thm:homo-mark}, it suffices to show the weak stabilization property and the moment condition for $f$.

By definition, $D_\zero f(W)$ is the number of induced subgraphs in $G(\cP \cup \{\zero\})|_W$ containing the vertex $\zero$ isomorphic to $A$,
\[
	D_\zero f(W) =  \sum_{x_{[k]} \subset \cP|_W }\I(\Gamma(\zero,x_{[k]}) \simeq A).
\]
Here recall that $\I$ denotes the indicator function. Thus, the functional $f$ is weakly stabilizing because almost surely,
\begin{equation} \label{dofna}
D_\zero f ({W_n}) \to  \Delta:=   \sum_{x_{[k]} \subset \cP} \I(\Gamma(\zero,x_{[k]}) \simeq A).
\end{equation}
Moreover, since $A$ is feasible, Lemma~\ref{lem:finite} implies that the limit $\Delta$ is finite (almost surely) and non-trivial.

For the moment condition, observe that
\begin{equation}\label{3-moment}
\sup_{\zero \in W\text{:cube}}\Ex[|D_\zero f (W)|^3] \leq \Ex[\Delta^3].
\end{equation}
Thus, our remaining task is to show that $\Ex[\Delta^3]$ is finite. Similar to the proof of Lemma~\ref{lem:covariance} (see also \cite[Lemma~3.4]{Grygierek-2019}), we see that  there exist constants $\{C(k,r,s,t): 0\leq r, s, t \leq k \}$, such that
\begin{align}
&& \Ex[\Delta^3]  \\
&=& \sum_{r=0}^{k-1} \sum_{s=0}^{k-1} \sum_{t=0}^{\min\{k-r,k-s\}}  C(k,r,s,t) \int_{(\R^d)^{\ell}} p_A(y_{[k]},z_{[k-r]},w_{[k-u]}) dy_{[k]} dz_{[k-r]} dw_{[k-u]}, \notag
\end{align}
where $\ell=3k-r-u$, $u=s+t$ and $p_A(y_{[k]},z_{[k-r]},w_{[k-u]})$ is the probability that the following three events happen 
\[
\Gamma (\zero, y_{[k]}) \simeq A,\quad  \Gamma(\zero,y_{[r]}, z_{[k-r]}) \simeq A, \quad \Gamma(\zero,y_{[s]},z_{[t]},w_{[k-u]}) \simeq A.
\]
In addition, each integral in the above expression is finite, which can be proved  in the same way as in Lemma~\ref{lem:finite}. Therefore $\Ex[\Delta^3]< \infty$. The proof is complete.
\end{proof}

The following result on component counts was shown in \cite{Last-Nestmann-Schulte-2018} by a different approach for which the rate of convergence in the CLT was also known. The multidimensional CLT itself can be easily derived from Theorem~\ref{thm:homo-mark} here.
\begin{theorem}[\cite{Last-Nestmann-Schulte-2018}]\label{thm:component}
\begin{itemize}
	\item[\rm(i)]
	Let $A$ be a feasible graph on $(k+1)$ vertices. Let $\zeta_n^{(A)}$ be the number of components in $G(\cP_n)$ isomorphic to $A$. Then as $n \to \infty$, 
	\begin{align*}
		\frac{\zeta_n}{n} \to &\frac{\lambda^{k+1}}{(k+1)!} \idotsint_{(\R^d)^k} \Prob( \Gamma(x_0 = \zero, x_1, \dots, x_k) \simeq A) \\
		&\quad \times \exp\left( \int_{\R^d} \left[ \prod_{i = 0}^k (1 - \varphi(y - x_i))  - 1 \right] dy\right) dx_1 \cdots dx_{k} >0,
	\end{align*}
in probability.

\item[\rm (ii)] Let $\{A_1, \dots, A_m\}$ be feasible connected graphs. Then
	\[
		\left( \frac{ \zeta_n^{(A_i)}  - \Ex[\zeta_n^{(A_i)}] }{\sqrt n} \right)_{i = 1}^m  \dto \Normal(0, \Sigma),
	\]
with explicit formula for $\Sigma$.

\end{itemize}

\end{theorem}

\subsection{Betti numbers}
For a bounded subset $W \subset \R^d$, denote by  $\cX_W$ the clique complex of the graph $G(\cP)|_W$, that is, the abstract simplicial complex formed by the cliques (or complete subgraphs) of $G(\cP)|_W$. (A simple example of the clique complex of a graph is given in Figure~\ref{f2}.)
Let $\beta_k(W)$, or $\beta_k(\cX_W)$ be the $k$th Betti number of the simplicial complex $\cX_W$. We are going to establish a LLN and a CLT for $\beta_k(W)$ as $W \to \R^d$.

\begin{figure}
\centering
\begin{subfigure}[b]{0.3\textwidth}
\fbox{\includegraphics[width=\textwidth]{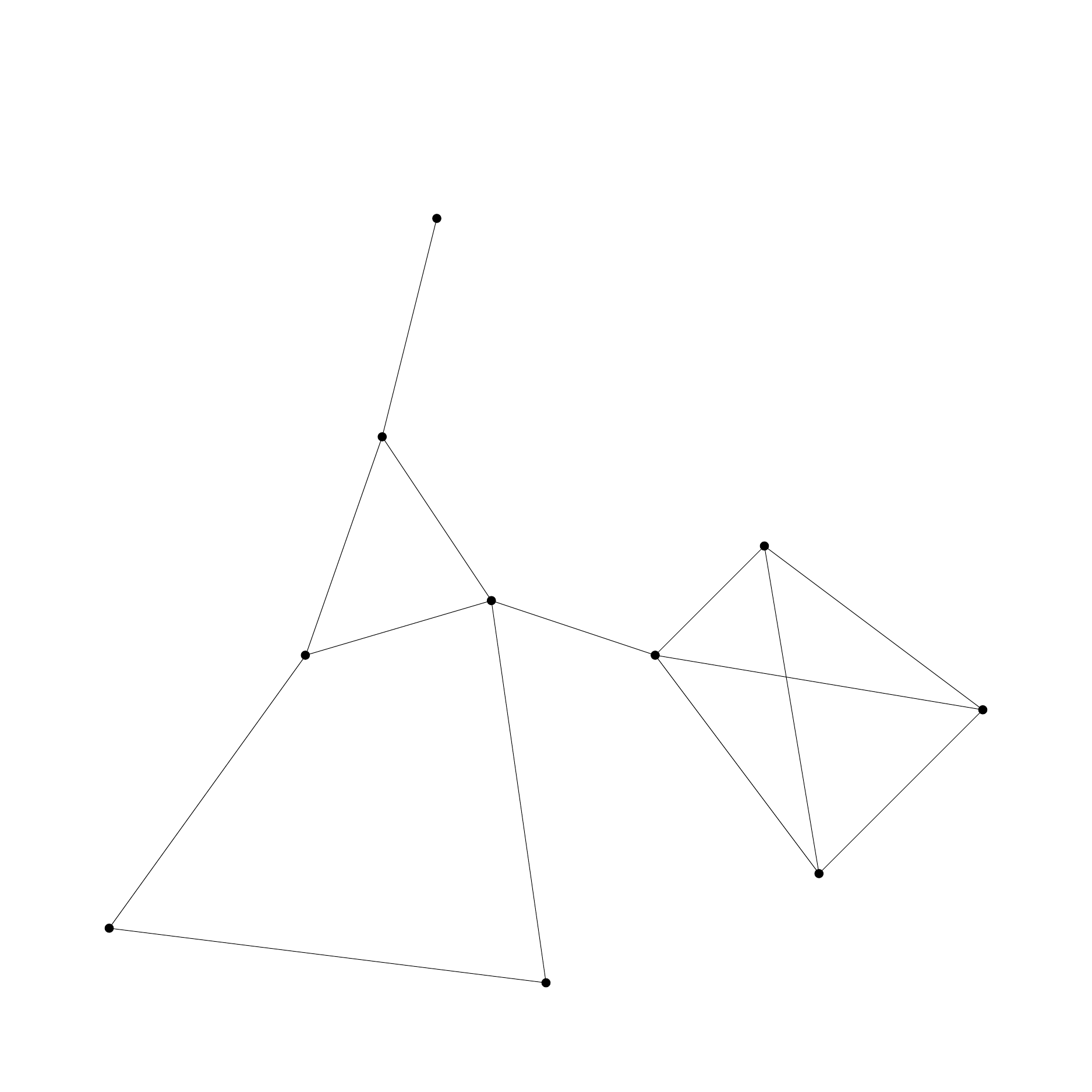}}
\caption{A graph}
\end{subfigure}
\quad
\begin{subfigure}[b]{0.3\textwidth}
\fbox{\includegraphics[width=\textwidth]{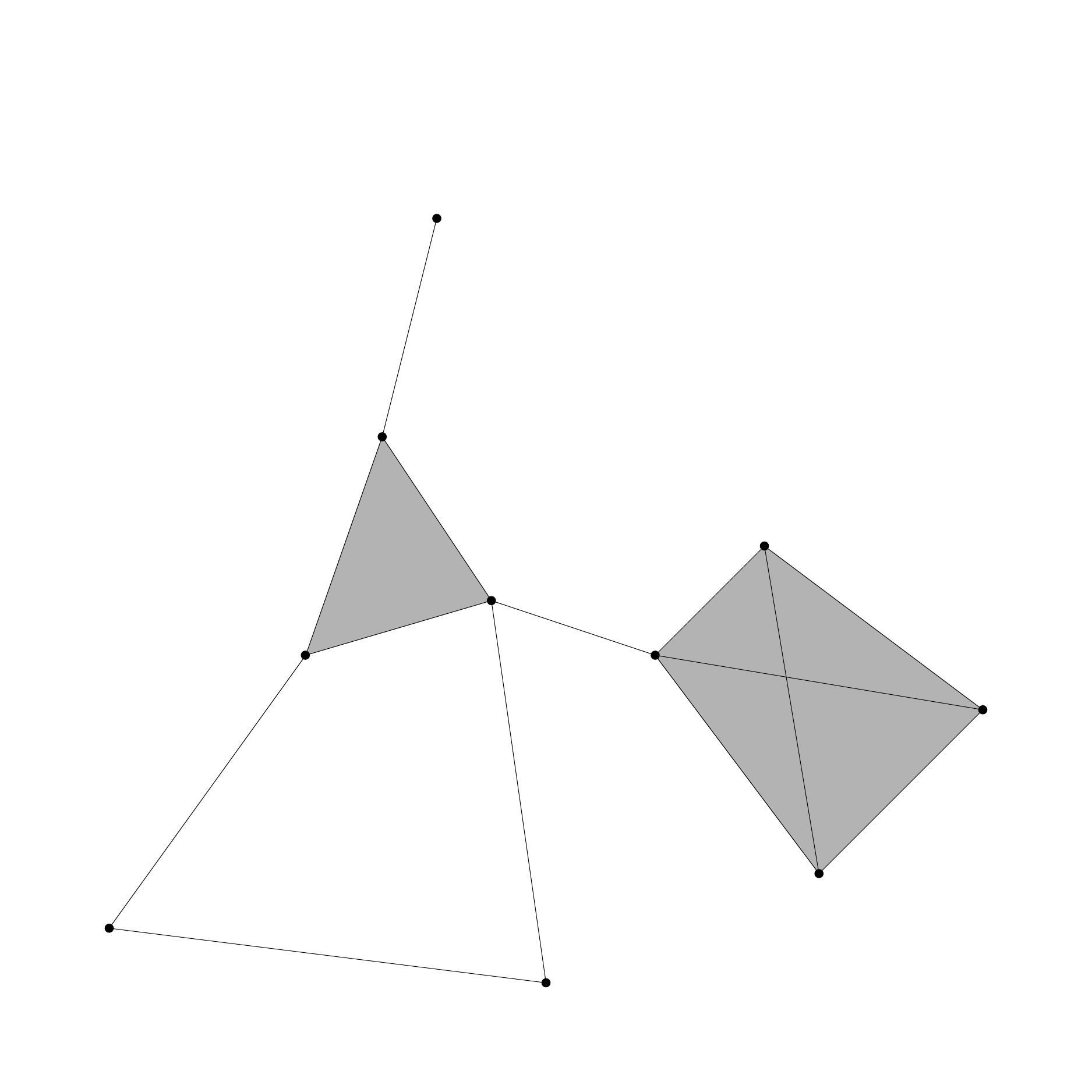}}
\caption{An intermediate}
\end{subfigure}
\quad
\begin{subfigure}[b]{0.3\textwidth}
\fbox{\includegraphics[width=\textwidth]{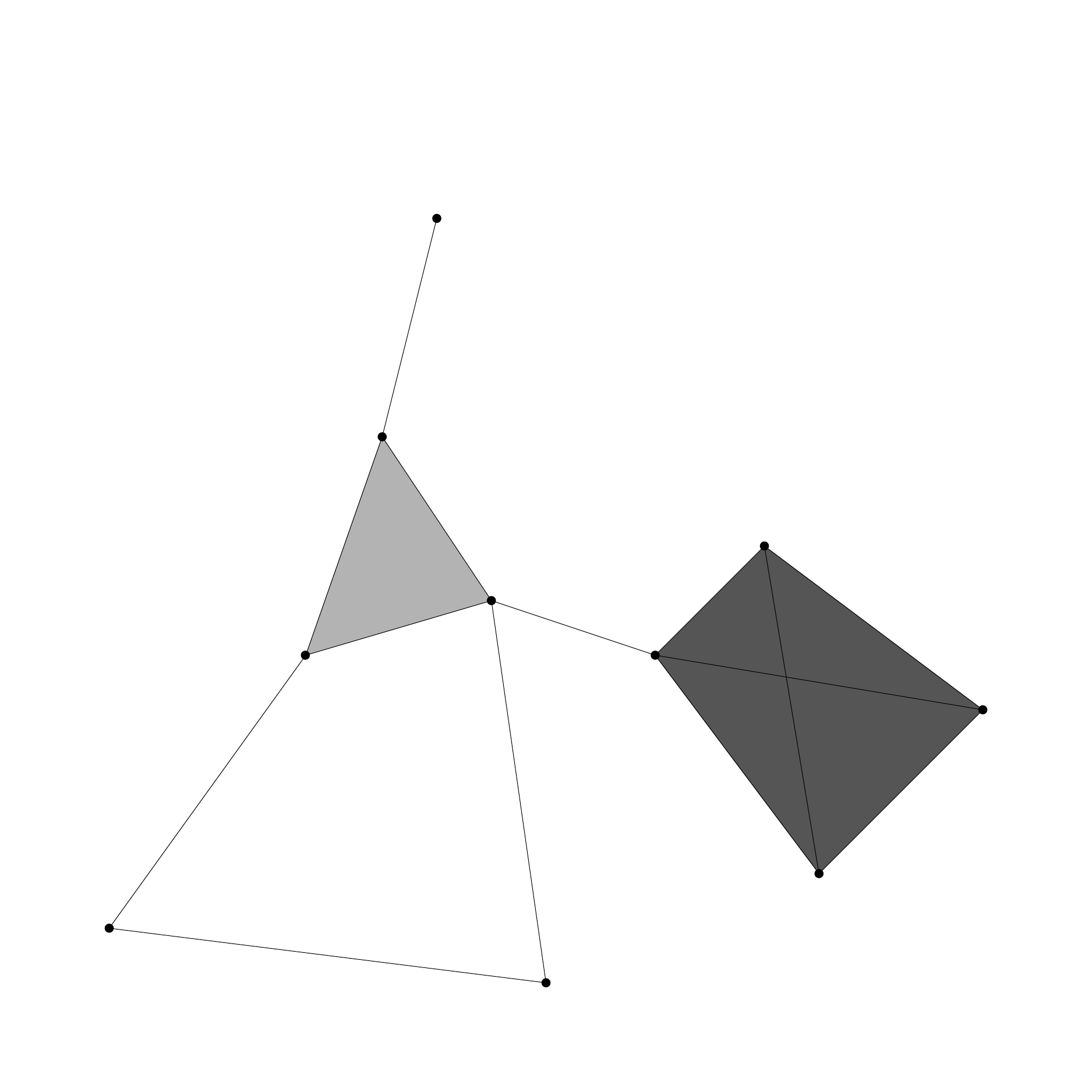}}
\caption{The clique complex}
\end{subfigure}
\caption{The clique complex of a graph.}
\label{f2}
\end{figure}

Let us give a quick review on Betti numbers and some necessary properties needed in the arguments. We refer the readers to the book \cite{Munkres-1984} for more details.

Let $\cK$ be an abstract simplicial complex, that is, a collection of nonempty subsets of a finite set $V$ closed under inclusion relation. An element $\sigma \in \cK$ is called a simplex and more precisely, a $k$-simplex, if $|\sigma| = k+1$.
For each $k$, denote by $\cK_k$ the set of all $k$-simplices in $\cK$, and let 
\[
	C_k(\cK) = \bigg\{\sum \alpha_i \lr{\sigma_i} : \alpha_i \in \Fi, \sigma_i \in \cK_k \bigg \}
\]
be a vector space on some fixed field $\Fi$, where $\lr{\sigma}$ denotes the oriented simplex. For $k \ge 1$, the boundary operator $\partial_k \colon C_k(\cK) \to C_{k - 1} (\cK)$ is defined as a linear mapping with
\[
	\partial_k(\lr{v_0, \dots, v_k}) = \sum_{i = 0}^k (-1)^i \lr{v_0, \dots, \hat{v}_i, \dots, v_k},
\]
on any oriented $k$-simplex $\lr{v_0, \dots, v_k}$. Here the symbol ${\hat{~}}$ over $v_i$ indicates that the vertex $v_i$ is removed from the sequence. (The operator $\partial_0 \colon C_0(\cK) \to \{0\}$ is defined to be a trivial one.) We can easily check that $\partial_k \circ \partial_{k+1} = 0$, and thus $B_k(\cK) := \Image \partial_{k + 1}  \subset Z_k (\cK) := \ker \partial_{k}$. The two are called the $k$th boundary group and the $k$th cycle group, respectively. The quotient space 
\[
	H_k (\cK) = Z_k(\cK) / B_k(\cK)
\]
is called the $k$th homology group of $\cK$, and its rank is the $k$th Betti number, 
\[
	\beta_k(\cK) = \rank H_k(\cK) = \dim Z_k(\cK) - \dim B_k(\cK).
\]
Note that the zeroth Betti number coincides with the number of connected components in the undirected graph $G= (V, E)$, where $E =\cK_1$.

Let $\{\cK^{(i)}\}_{i\in I}$ be a finite collection of disjoint simplicial complexes. Then the disjoint union $\sqcup_{i \in I}\cK^{(i)} $ is again  a simplicial complex, and the following identity holds
\begin{equation}\label{additive}
	\beta_k \Big(\bigsqcup_{i \in I}\cK^{(i)} \Big) = \sum_{i \in I} \beta_k(\cK^{(i)}).
\end{equation}
This property follows directly from the definition. Another useful property is the following. For two finite simplicial complexes $\cK \subset \tilde \cK$, and any $k \ge 0$, 
	\begin{equation}\label{Betti-estimate}
		| \beta_k(\cK) - \beta_k (\tilde\cK) | \le  \sum_{j = k}^{k + 1} (S_j(\tilde \cK) - S_j(\cK)),
	\end{equation}
where $S_j(\cK)$ (resp.~$S_j(\tilde \cK)$) denotes the number of $j$-simplices in $\cK$ (resp.~$\tilde \cK$).
The proof of this inequality can be found in \cite{Trinh-2017, ysa}.

The following LLN for Betti numbers is analogous to a LLN for Betti numbers in the thermodynamic regime \cite{Goel-2019, ysa}.

\begin{theorem}
	As the sequence of cubes $\{W_n\}$ tends to $\R^d$, 
	\[
		\frac{\beta_k(W_n)}{|W_n|} \to \bar \beta_k \quad \text{in probability,}
	\]
where $\bar \beta_k$ is a constant. The limit $\bar \beta_k$ is positive, if $\varphi \in (0,1)$.  
\end{theorem}
\begin{proof}
	We will only show the convergence of the mean, because the convergence in probability is a consequence of the CLT below. It suffices to consider the sequence of cubes $W_n = [-n^{1/d}/2, n^{1/d}/2)^d$ as $n \to \infty$. For $L > 0$, divide the cube $W_n$ according to the lattice $L^{1/d} \Z^d$ and let $\{C_i\}_{i = 1}^{\ell_n}$ be the lattice cubes entirely contained in $W_n$. It is clear that $\ell_n /n \to 1/L$ as $n \to \infty$. Let 
	\[
		\cK = \bigsqcup_{i = 1}^{\ell_n} \cX_{C_i}
	\]
be the disjoint union of $\{\cX_{C_i}\}$ which is a subcomplex of $\cX_{W_n}$. It follows from the estimate~\eqref{Betti-estimate} that
\[
	|\beta_k(\cX_{W_n}) - \beta_k(\cK) | \le \sum_{j = k}^{k+1}\left( S_j(\cX_{W_n}) - S_j(\cK) \right).
\]
Then by taking the expectation, we obtain that 
\[
	\left| \frac{\Ex[\beta_k(\cX_{W_n})]} {n} - \frac1n \sum_{i = 1}^{\ell_n} \Ex[\beta_k(\cX_{C_i})]\right| \le  \sum_{j = k}^{k+1}\left( \frac{\Ex[S_j(\cX_{W_n})]}{n} - \frac{1}{n} \sum_{i = 1}^{\ell_n} \Ex[S_j(\cX_{C_i})] \right).
\]
Here we have used the fact that $\cK$ is the disjoint union of $\{\cX_{C_i}\}$. In addition, note that all $\cX_{C_i}$ have the same distribution. Therefore 
\[
	\left| \frac{\Ex[\beta_k(\cX_{W_n})]} {n} - \frac {\ell_n} n  \Ex[\beta_k(\cX_{C_1})]\right| \le  \sum_{j = k}^{k+1}\left( \frac{\Ex[S_j(\cX_{W_n})]}{n} - \frac{\ell_n}{n}  \Ex[S_j(\cX_{C_1})] \right).
\]
By letting $n\to \infty$, it follows that 
\begin{align*}
	&\limsup_{n \to \infty} \frac{\Ex[\beta_k(\cX_{W_n})]} {n}  \le \frac{\Ex[\beta_k(\cX_{C_1})]}{L} +\sum_{j = k}^{k+1} \left(\lim_{n \to \infty} \frac{\Ex[S_j(\cX_{W_n})]}{n}  - \frac{ \Ex[S_j(\cX_{C_1})] }{L} \right),\\
	&\liminf_{n \to \infty} \frac{\Ex[\beta_k(\cX_{W_n})]} {n}  \ge \frac{\Ex[\beta_k(\cX_{C_1})]}{L} -\sum_{j = k}^{k+1} \left(\lim_{n \to \infty} \frac{\Ex[S_j(\cX_{W_n})]}{n}  - \frac{ \Ex[S_j(\cX_{C_1})] }{L} \right),
\end{align*}
and hence 
\begin{align*}
	\limsup_{n \to \infty} \frac{\Ex[\beta_k(\cX_{W_n})]} {n}  - \liminf_{n \to \infty}& \frac{\Ex[\beta_k(\cX_{W_n})]} {n}  \\
	&\le 2 \sum_{j = k}^{k+1} \left(\lim_{n \to \infty} \frac{\Ex[S_j(\cX_{W_n})]}{n}  - \frac{ \Ex[S_j(\cX_{C_1})] }{L} \right).
\end{align*}
Since $S_j$ counts the number of complete subgraphs on $(j+1)$ vertices, Lemma~\ref{lem:covariance}(i) ensures that the limit of ${\Ex[S_j(\cX_{W_n})]}/{n}$ exists. This also implies that the right hand side of the above equation goes to zero as $L \to \infty$. Therefore, the limit $\lim_{n \to \infty}{\Ex[\beta_k(\cX_{W_n})]} /{n}$ exists. We will show the positivity of $\bar \beta_k$ at the end of this section. The proof is complete.
\end{proof}

Next, we establish a CLT for Betti numbers. Related results are CLTs for Betti numbers and persistent Betti numbers in \cite{hst, ysa}, respectively.
\begin{theorem}
As the sequence of cubes $W$'s tends to $\R^d$, 
	\[
		\frac{\beta_k(W) - \Ex[\beta_k(W)]}{\sqrt {|W|}} \dto \Normal(0, \sigma^2_k),
	\]
for a constant $\sigma^2_k \ge 0$. The limiting variance $\sigma^2_k$ is positive, if $\varphi \in (0,1)$.
\end{theorem}

This is again an application of Theorem~\ref{thm:homo-mark}. Thus, we need to show the following 
\begin{itemize}
	\item[(i)] Betti numbers are weakly stabilizing; 
	\item[(ii)] the moment condition holds;
	\item[(iii)] and the limit add-one cost is non-trivial, if $\varphi \in (0,1)$.
\end{itemize}
The moment condition follows immediately from that for subgraph counts, and hence the proof is omitted. We now show the weak stabilization and the non-triviality in sequent.

\begin{lemma}
Let $\{\cK^{(n)}\}_{n = 1}^\infty$ be a sequence of increasing simplicial complexes. Assume that $K_0$ is a finite set of complexes which is disjoint from $\cK^{(n)}$ such that $\tilde \cK^{(n)} := \cK^{(n)} \sqcup K_0$ is also a simplicial complex for all $n$.
Then the following limit exists
	\[
		\lim_{n \to \infty } (\beta_k(\tilde \cK^{(n)}) - \beta_k(\cK^{(n)})).
	\]
\end{lemma}

\begin{proof}
	From the definition of Betti numbers, we can write 
\begin{align*}
	\beta_k(\tilde\cK^{(n)}) - \beta_k(\cK^{(n)}) &= \left\{\dim \tilde Z_k^{(n)} - \dim Z_k^{(n)} \right\}  + \left\{ \dim \tilde Z_{k+1}^{(n)}  - \dim Z_{k+1}^{(n)} \right\} \\
	&\quad - \left\{\dim \tilde C_{k+1}^{(n)} - \dim C_{k+1}^{(n)} \right\}.
\end{align*}
Here we use the superscript ${}^{(n)}$ and that with the symbol $\tilde{}$ to indicate quantities of $\cK^{(n)}$ and $\tilde \cK^{(n)}$, respectively. It follows from the assumption $\tilde \cK^{(n)} = \cK^{(n)} \sqcup K_0$ that the difference   $(\dim \tilde C_{k+1}^{(n)} - \dim C_{k+1}^{(n)}) $ is a constant.

	Let $\partial_k \colon \tilde C_k^{(n+1)} \to \tilde C_{k-1}^{(n+1)}$ denote the boundary operator for $\tilde \cK^{(n+1)}$. Since $\cK^{(n)}, \tilde \cK^{(n)}$ and $\cK^{(n+1)}$ are sub-complexes of $\tilde \cK^{(n+1)}$, we get that 
	\[
		Z_k^{(n)}  = \ker \partial_k \cap C_k^{(n)}, \quad \tilde Z_k^{(n)}  = \ker \partial_k \cap \tilde C_k^{(n)}, \quad Z_k^{(n+1)}  = \ker \partial_k \cap C_k^{(n+1)}.
	\]
As subspaces of $\tilde C_k^{(n+1)}$, we can easily check the relation 
\[
	\tilde C_k^{(n)} \cap C_k^{(n+1)} = C_k^{(n)},
\]
from which we deduce that
\[
	\tilde Z_k^{(n)} \cap Z_k^{(n+1)} = Z_k^{(n)}.
\]
It then follows that 
\begin{align*}
	\dim Z_k^{(n)} &= \dim \tilde Z_k^{(n)}  + \dim Z_k^{(n+1)} - \dim ( \tilde Z_k^{(n)} + Z_k^{(n+1)}) \\
	&\ge \dim \tilde Z_k^{(n)}  + \dim Z_k^{(n+1)}  - \dim  \tilde Z_k^{(n+1)}.
\end{align*}
This implies the increasing property of the sequence $\{\dim \tilde Z_k^{(n)}  - \dim Z_k^{(n)}  \}_{n}$. Since the roles of $k$ and $k+1$ are equal, we conclude that $\beta_k(\tilde\cK^{(n)}) - \beta_k(\cK^{(n)})$ is an increasing sequence. In addition, it is bounded by taking into account of the inequality \eqref{Betti-estimate}. Therefore, the limit exists, which completes the proof.
\end{proof}

\begin{lemma}
	$\beta_k$ is weakly stabilizing.
\end{lemma}
\begin{proof}
Let $W_n$ be a sequence of increasing cubes tending to $\R^d$. Let $\omega \in \Omega$ be such that the set $\cP$ is locally finite and that the graph $G(\cP \cup \{\zero\})$ has a finite number of edges connected to $\zero$. Note that the set of such $\omega$ has probability one. Then there is a number $N$ (depending on $\omega$) such that for $n \ge N$, $W_n$ contains all vertices connected to $\zero$. Let $\cK^{(n)} = \cX_{W_n}$ and let $\tilde \cK^{(n)}$ be the clique complex of the graph $G(\cP \cup \{\zero\})|_{W_n}$. Then for $n \ge N$, $K_0 = \tilde \cK^{(n)} \setminus \cK^{(n)}$ does not change. By definition of the add-one cost, it holds that 
\[
	D_\zero \beta_k(W_n) = \beta_k(\tilde \cK^{(n)}) - \beta_k(\cK^{(n)}),
\]
from which the weak stabilization follows from the above lemma.
\end{proof}

\textbf{On the positivity of $\bar \beta_k$ and $\sigma_k^2$.} Let $O_k$ be the graph on $[2k+2]$ with all except the following edges $\{\{1, k+1\}, \{2, k+3\}, \dots, \{k, 2k+2\}\}$. The clique complex $\cX_{O_k}$ is a boundary of the $(k+1)$-dimensional cross-polytope (Definition 3.3 in \cite{Kahle-2011}). It was known that \cite{Kahle-2009} $\beta_k(\cX_{O_k}) = 1$, and that $\beta_k(\cX_A) = 0$ for any graph $A$ on less than $2k+2$ vertices.

\begin{lemma}
Assume that the graph $O_k$ is feasible. Then $\bar \beta_k > 0$ and $\sigma_k^2 > 0$. In particular, if $\varphi \in (0,1)$, then $O_k$ is feasible and hence, both $\bar \beta_k $ and $\sigma_k^2$ are positive.
\end{lemma}

\begin{proof}
Assume that $O_k$ is feasible. Recall that $\zeta_n^{(O_k)}$ denotes the number of components in $G(\cP_n)$ isomorphic to $O_k$. It follows from the property~\eqref{additive} that, 
\[
	\beta_k(W_n) \ge \zeta_n^{(O_k)},
\]
and then from Theorem~\ref{thm:component} that
\[
	\bar \beta_k \ge \lim_{n\to \infty} \frac{\zeta_n^{(O_k)}}{n} > 0.
\]

Next, for the positivity of the limiting variance, we will show that $\Delta$ is non-trivial. Let $\Omega_0$ be the event that the component containing $\zero$ in $G(\cP \cup \{\zero\})$ is isomorphic to $O_k$. Then by using the multivariate Mecke equation, we can prove that
\begin{align*}
	\Prob(\Omega_0) &= \frac{\lambda^{2k+1}}{(2k+1)!}\idotsint_{(\R^d)^{2k+1}} \psi_{O_k}(\zero, x_{[2k+1]}) \\
	&\quad \times \exp \bigg(\int_{\R^d} \bigg[ \prod_{i=0}^{2k+1} (1 - \varphi(y - x_i))  - 1 \bigg] dy\bigg) dx_{[2k+1]}  > 0.
\end{align*}
(See also Proposition~3.1 in \cite{Last-Nestmann-Schulte-2018}.)
On $\Omega_0$, when the cube $W$ is large enough, 
\[
D_\zero \beta_k (W) = \beta_k (\cX_{\Gamma(\zero, x_{[2k+1]})}) - \beta_k (\cX_{\Gamma(x_{[2k+1]})}) = \beta_k (\cX_{O_k}) >0.
\]
Therefore, $\Delta(\omega) >0$ on $\Omega_0$, that is, $\Delta$ is non-trivial. The proof is complete.
\end{proof}

\section{Size of the biggest component}

In this section, we aim to prove the central limit theorem for the size of the biggest cluster of $G(\cP|_W)$, or of $G(\cP)|_{W}$ as $W \rightarrow \R^d$ under some conditions on the connection function $\varphi$ as follows.  We suppose  that there exists $\phi \colon \R_+ \mapsto [0,1]$, such that  $\varphi(x)=\phi(|x|)$ for all $x \in \R^d$ and 
  \begin{itemize}
  	\item [(C1)] the function $\phi$ is continuous at $0$ and $\phi(0)=1$,
  	\item[(C2)] there exist positive constants $C_0, \epsilon_0$ such that for all $r> 0$,
  	\begin{equation} \label{cond:phi}
  	\phi(r) \leq C_0 r^{-(5d+\epsilon_0)}. 
  	\end{equation}
  \end{itemize}

For any cube $W \subset \R^d$, we denote the biggest connected component of $G(\cP)|_W$ (resp.\ of $G(\cP|_W)$), that is, the connected component with the largest number of vertices, by $\cC(W)$ (resp.\ $\cC(\cP|_W)$). When there are more than one biggest components, we choose $\cC(W)$ to be the component having the vertex with smallest coordinate in the lexicographic order. We will need the following result on the uniqueness of the infinite cluster in the graph $G(\cP)$.

\begin{lemma}[{\cite[Section 6.4]{MR}}] Assume that $\int_{\R^d} \varphi(x)dx \in (0, \infty)$. Then there is a critical parameter $\lambda_c \in (0, \infty)$ such that when $\lambda \in (0, \lambda_c)$, all the connected components of the random connection model are finite a.s., whereas when $\lambda>	\lambda_c$ the random graph has a unique infinite connected component.
\end{lemma}  

\begin{theorem} \label{thm:size}
	Assume that the conditions (C1)-(C2) hold. 
	 Then there exists $\lambda^\star \in (\lambda_c, \infty)$, such that for any fixed $\lambda > \lambda^\star$, as the sequence of cubes $W$'s tends to $\R^d$,
	 \[
	\frac{|\cC(W)| - \Ex[|\cC(W)|]}{\sqrt {|W|}} \dto \Normal(0, \sigma^2),
	\]
where $\sigma^2 = \sigma^2(\lambda) >0$.	
\end{theorem}
\begin{proof} 
Let us begin with an expression for the add-one cost. Let $W$ be a cube containing the origin $\zero$. Recall that $\cC(W)$ denotes the biggest connected component in $G(\cP)|_{W}$. Note that $G(\cP \cup \{\zero\}) |_{W}$ is obtained from $G(\cP)|_{W}$ by adding one vertex $\zero$ and edges from $\zero$. To identify the biggest component  in $G(\cP \cup \{\zero\}) |_{W}$ which is denoted by $\cC'(W)$, there are three cases to consider.

Case 1: the vertex $\zero$ is  connected to $\cC(W)$. Then it is clear that the biggest component in $\cC'(W)$ is the connected component containing $\cC(W)$. Thus, the add-one cost is written as 
\[
	\Delta(W) := D_\zero f(W) = |\cC'(W)| - |\cC(W)| = 1 + \#\{x \in \cP|_W : x \overset{W}{\leftrightarrow} \zero, x \not \in \cC(W) \}.
\]
Here $x \overset{W}\leftrightarrow o$ means there is a path $\gamma =(x_i)_{i=0}^l \subset \cP|_W$, such that from $x=x_0, x_l =o$  and $x_i \sim x_{i+1}$ for all $i=0,\ldots,l-1$. For simplicity, we write $x \leftrightarrow o$ in case $x \overset{\R^d}\leftrightarrow o$.

Case 2: the vertex $\zero$ is not connected to $\cC(W)$ and the new component containing $\zero$ becomes the biggest one. In this case, 
\begin{align*}
	\Delta(W) &= 1 + \#\{x \in \cP|_W : x \overset{W}\leftrightarrow \zero \} - |\cC(W)|\\
	&= 1 + \#\{x \in \cP|_W : x \overset{W}\leftrightarrow \zero, x \not \in \cC(W) \} - |\cC(W)|.
\end{align*}

Case 3: the vertex $\zero$ is not connected to $\cC(W)$ and the new component containing $\zero$ has size smaller than $\cC(W)$.  When it happens, then  $\cC'(W) = \cC(W)$, and hence $\Delta(W) = 0$.

To apply Theorem \ref{thm:homo-mark}, we will show the weak stabilization and a moment condition with $p = 3$,
\begin{equation} \label{edt3}
	\sup_{\zero \in W: \text{cube}} \E[\Delta(W)^3] < \infty. 
\end{equation}
Define
\begin{align}
	\Delta:  = \I(o \sim \cC(\R^d) ) \Big( 1 + \# \{x \in \cP :  x \leftrightarrow o, x \not \in \cC(\R^d) \} \Big),
\end{align}	
where $\cC(\R^d)$ is the unique infinite cluster in $G(\cP)$, and $\zero \sim \cC(\R^d)$ is the event that $\zero$ is connected to $\cC(\R^d)$. For the weak stabilization, we will prove in Subsection \ref{sss:d-i}  that as a sequence of cubes $W$'s tends to $\R^d$,
\begin{equation}\label{delta-infinity}
\Delta(W) \Pto \Delta.
\end{equation}  
The idea is that, with high probability, $\cC(W)$ belongs to $\cC(\R^d)$, and Case 2 does not happen. For the moment condition, by using an estimate that 
\[
	0 \le \Delta(W) \le 1 + \#\{x \in \cP|_W : x \overset{W}\leftrightarrow \zero, x \not \in \cC(W) \}. 
\]
showing the moment condition reduces to a problem of estimating the probability 
\[
	\Prob(x \overset{W}\leftrightarrow \zero, x \not \in \cC(W) ).
\] 
We will show it in the next sub-section.

Finally, note that when  $\lambda > \lambda_c$, we have $\Prob( \Delta \neq 0) = \Prob (o \sim \cC(\R^d) ) >0$. Hence by Theorem \ref{thm:homo-mark}, the limiting variance $\sigma^2 >0$, which completes the proof of Theorem~\ref{thm:size}.
\end{proof}

\begin{remark}
	We guess that Theorem \ref{thm:size} holds for all $\lambda > \lambda_c$. To reduce the condition that $\lambda > \lambda^\star$  to $\lambda > \lambda_c$, it appears to us that we need  some renormalization of Russo--Seymour--Welsh type, as done in \cite[Chapter 10]{Penrose-book} for random geometric graph. However, this task for general random connection models is more complicated and highly nontrivial, so we leave it for future research. 	  In fact, key tools in the proof of  the weak stabilization \eqref{delta-infinity} and the moment condition \eqref{edt3} are the renormalization steps to estimate the decay of the probability that there exists a long path not intersecting to the biggest cluster, see  in Proposition \ref{prop:theta} and Lemma \ref{lem:renorm}.   In this estimate, we need $\lambda$ to be large enough for an initial ingredient of the renormalization procedure, see in particular Lemma \ref{lem:comp} and the condtion \eqref{lbat}. 
\end{remark}

In the next subsection, renormalization estimates  and some preparations are proved. The proofs of the weak stabilization and the moment condition are then given in Subsection \ref{ssec:1-2}.

\subsection{ Renormalization and preliminaries} \label{ssec:renor}
  For each $\delta>0$, we tessellate the whole space $\R^d$ to cubes of size $\delta$ and denote the collection of cubes by $\Gamma$.  Let $G_{\delta}$ be the random graph obtained from $G(\cP)$ by deleting the edges between vertices in non-adjacent cubes (that is, keeping only edges between vertices in the same cube or in adjacent cubes). 
For each cube $B \in \Gamma$, when $B \cap \cP \neq \varnothing$, we choose an arbitrary point in $B \cap \cP \neq \varnothing$ (in some deterministic way), say $x_B$, to be the representation of $B$.  Let $\textrm{Per}(\delta)$ be the induced subgraph of  $G_{\delta}$ restricted on the vertex set  $V=\{x_B: \, B \in \Gamma,  B \cap \cP =\varnothing \}$. Then for each cube $\Lambda$, we define
  \[  \cC^{{\rm per}}_{\delta} (\Lambda) = \textrm{ the biggest cluster of $\rm{Per}(\delta)|_{\Lambda}$},\]
  and
     \[ \cC_{\delta} (\Lambda) = \textrm{ the connected component of  $G_{\delta}|_{\Lambda}$ containing $\cC^{{\rm per}}_{\delta} (\Lambda)$}. \]
     
For all $t \geq 2s >0$, define 
    \begin{align}
     \beta_{\delta}(t) &= \Prob(\cC_{\delta}(B_o(t)) \not \subset \cC_{\delta}(\R^d)),\\
     \nu_{\delta}(s, t) &=  \sup_{y: |y|_{\infty} \leq t-s } \Prob(\cC_{\delta}(B_y(s)) \not \subset \cC_{\delta}(B_o(t)) ).
      \end{align}
Here $|x|_\infty = \max_{1 \le i \le d} |x_i|$ denotes the infinity norm of $x = (x_1, \dots, x_d) \in \R^d$, and $B_x(t) = \{ y \in \R^d : |y - x|_\infty \le t\}$ denotes the closed ball of radius $t$ centered at $x$ with respect to the infinity norm. Notice that we have used $B_r(x)$ to denote the closed ball of radius $r$ centered at $x$ with respect to the Euclidean norm. In this section, for the simplicity of notation we denote the closed ball under the infinity norm by $B_x(r)$.

\begin{lemma} \label{lem:comp}
	Assume that the condition (C1) holds. 
	\begin{itemize}
		\item [\rm(i)]
		There exist positive constants $ c_0, \delta_0, \lambda_0$, such that for all $\lambda > \lambda_0$, and $t,s$ large enough satisfying $\tfrac{t}{2}\geq s \geq \sqrt{t} $, it holds that
	 \begin{equation*}
	 \nu_{\delta_0}(s,t) \leq \exp(-c_0s),
	 \end{equation*}
	 and
	  \begin{equation*}
	   \beta_{\delta_0} (t) +  \Prob(|\cC_{\delta_0}(B_o(t))| \leq c_0t^d)\leq \exp(-c_0t).
	  \end{equation*}
	  
\item[\rm(ii)] For any fixed $\delta$ and $\Lambda$, 
       \begin{equation*}
       \lim_{\lambda \rightarrow \infty} \Prob(|G(\cP|_{\Lambda})|=|\cC_{\delta}(\Lambda)|) =1.
       \end{equation*}
     \end{itemize}
 	  
\end{lemma}

We postpone the proof of Lemma \ref{lem:comp} to Appendix~\ref{appendix_proof}.

We define for $u \in \R^d$, and $\alpha, t>0$,
   \begin{equation}\label{def-akp0}
    \cA_{\kappa}(u, \alpha,t) =\{  \exists x, y \in \cP: |x-u|_{\infty} \leq 2t, x \sim y, |x-y|_{\infty} \geq \alpha t \}.
  \end{equation}
Note that the probability of $\cA_{\kappa}(u,\alpha,t)$ does not depend on the position of $u$, so we can define
   \begin{equation} \label{def-akp}
  \kappa(\alpha,t)= \Prob(\cA_{\kappa}(u,\alpha,t)).
  \end{equation}
It follows from the condition  (C2) that 
    \begin{align} \label{pb-akp}
    \kappa(\alpha, t)= \cO(1) \times \int \limits_{[-2t,2t]^d} dx  \int \limits_{y: |y-x|_{\infty} \geq \alpha t} \frac{ dy}{|x-y|^{5d+\epsilon_0}} =\cO(t^{-(3d+\epsilon_0)}),
    \end{align}
with $\epsilon_0$ as in that condition. Here $\cO$ is the big O notation with a constant not depend on $t$.

From now on, we fix $\delta = \delta_0$ as in Lemma \ref{lem:comp} and omit $\delta$ in the notation of $\nu$ and $\beta$. Define for $J \subset [d]$, with $[d] = \{1,\ldots,d\}$, $x \in \R^d, t>s>0$,
\begin{align}
\cA_{\theta_J}(x,s,t) = \Big \{  \exists \, \gamma = (x_i)_{i=0}^l \subset \cP:  x_i \sim x_{i+1}, x_i \in B^J_x(t) \setminus \cC_{\delta}(B^J_x(t)),  \notag \\   
  (i=0,\ldots, l-1), x_0 \in B^J_x(s), \,  x_l \in B^J_x(2t)  \setminus B^J_x(t)   \Big \}, \label{ajst}
\end{align}
where 
  \begin{equation} \label{bxjr}
  B^J_x(r) = x + \prod_{j \in J} [-r,r] \times \prod_{j \in [d] \setminus J} [0,2r].
  \end{equation}
Here we also omit $\delta$ in the notation. Notice that by the translation invariance and the rotation invariance, $\Prob(\cA_{\theta_J}(x,s,t) ) {=} \Prob(\cA_{\theta_{J'}}(o,s,t))$ for all $x \in \R^d$, and all $J, J'$ with $|J|=|J'|$. Thus,  we can denote   
 \begin{equation*}
 \theta_j(\tfrac{1}{16},t) = \Prob (\cA_{\theta_J}(x,\tfrac{t}{16},t)),
 \end{equation*}
 for any $J$ with  $|J|=j \leq d$ and $x \in \R^d$. Define also
   \[ \theta(\tfrac{1}{16}, t) = \max_{0\leq j \leq d} \theta_j(\tfrac{1}{16},t). \]
   
\begin{figure}[thb]
\centering
	\includegraphics[width = 12cm]{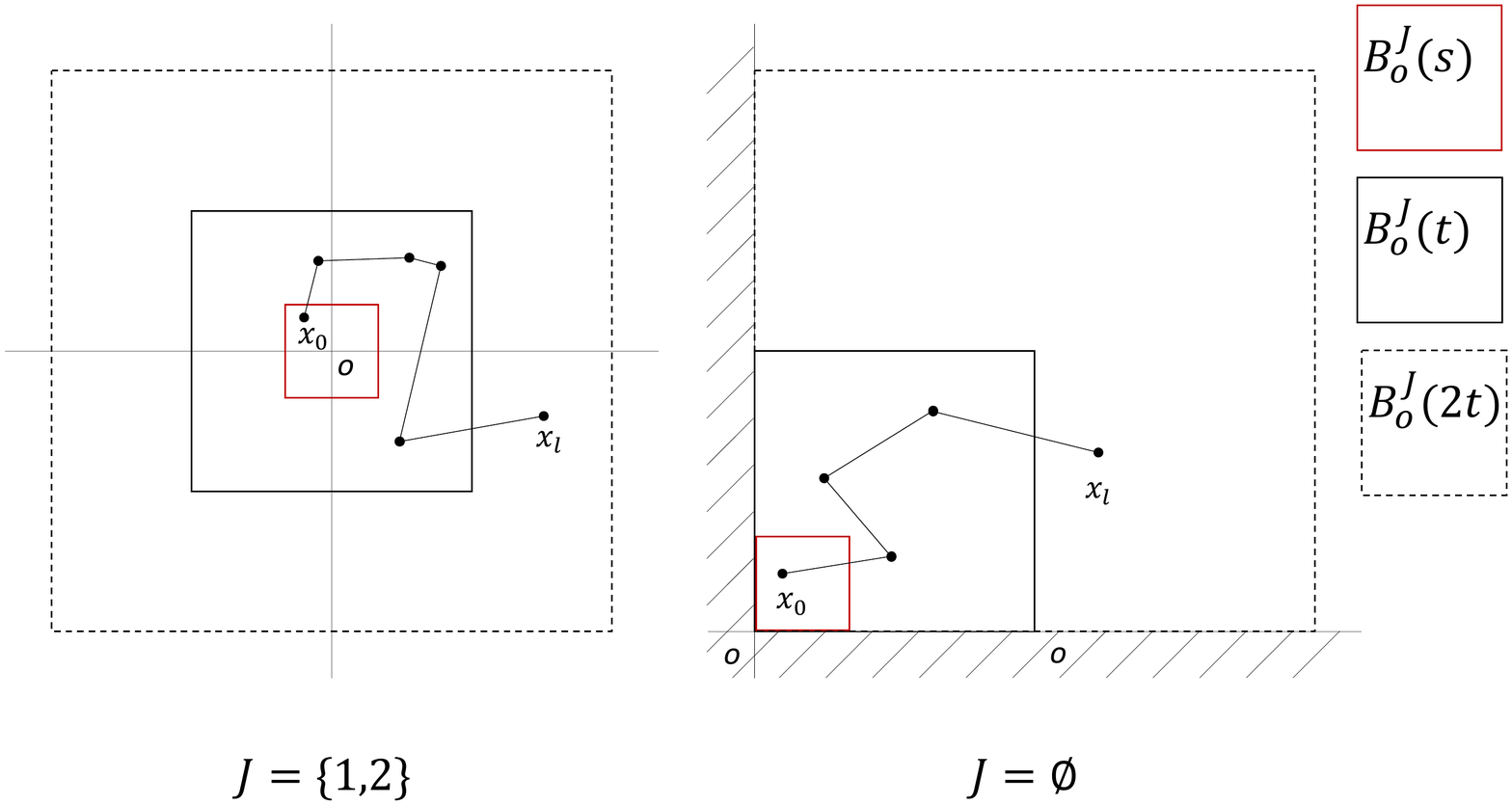}
	
	\vspace{-2cm}
	
	\includegraphics[width = 12cm]{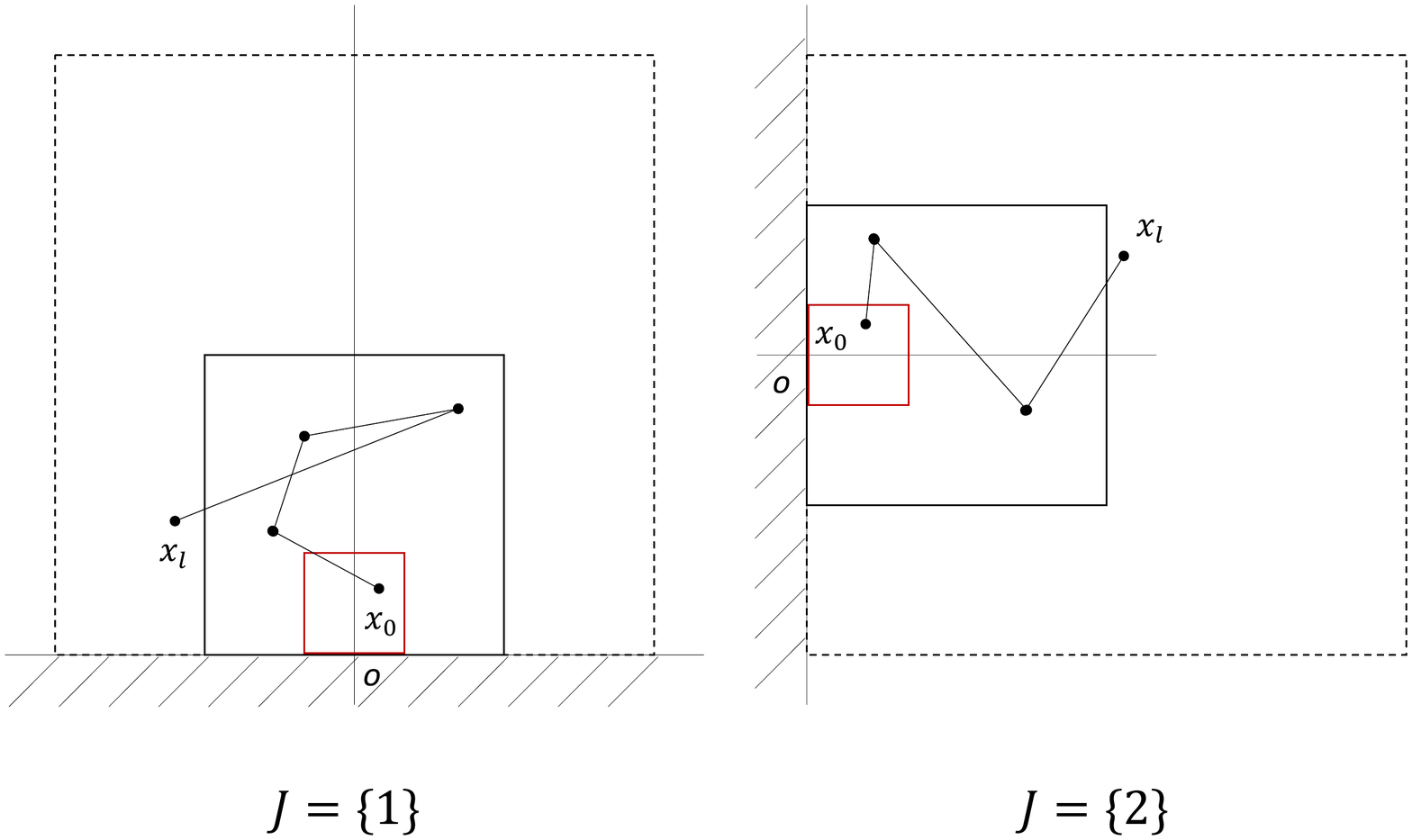}
	\caption{Illustration of events $\cA_{\theta_J}$}
\end{figure}

\begin{proposition} \label{prop:theta}
 	Suppose that the conditions (C1) and (C2) hold. Then there exist  positive constants $\lambda_1, C_1$, such that when $\lambda > \lambda_1$, for all $j=0,\ldots,d$,
 	  \begin{equation*}
 	  \theta(\tfrac{1}{16},t) \leq C_1 t^{-(3d+\epsilon_0)},
 	  \end{equation*}
 	  with $\epsilon_0$ as in (C2).
 \end{proposition}
The proof of Proposition~\ref{prop:theta} is given in Sub-section~\ref{ssec:tht}.

\begin{remark} It would be more natural if we can replace $\cC_{\delta}(B^J_x(t))$ by the biggest cluster $\cC(B^J_x(t))$ in the definition of $\cA_{\theta_J}(x,s,t)$ in  \eqref{ajst}. However, the proof of Proposition~\ref{prop:theta} requires some prior estimates as in Lemma \ref{lem:comp}, which are currently not available. More precisely, Proposition~\ref{prop:theta} still holds if we substitute the family of connected components  $\{\cC_{\delta}(B_x(t))\}_{t\geq 1}$ by any other family $\{\cC'(B_x(t))\}_{t\geq 1}$ satisfying 
	 \begin{align*}
	\beta'(t) &= \Prob(\cC'(B_o(t)) \not \subset \cC'(\R^d)) \leq \exp(-c(\log t)^2),\\
	\nu'(s, t) &=  \sup_{y: |y|_{\infty} \leq t-s } \Prob(\cC'(B_y(s)) \not \subset \cC'(B_o(t)) ) \leq \exp(-c(\log s)^2),
	\end{align*}
with $c$ a positive constant, for all  $t, s$ large real numbers such that $ \tfrac{t}{2}\geq s \geq \sqrt{t}$, and satisfying the inequality \eqref{lbat} (which is a consequence of Lemma \ref{lem:comp} (ii) when considering $\{\cC_{\delta}(B_o(t))\}_{t\geq 1}$). We could not prove directly these properties for the family of biggest clusters  $\{\cC(B_x(t))\}_{t\geq 1}$. Instead, we show in Lemma \ref{lem:comp} that the properties hold for  $\{\cC_{\delta}(B_x(t))\}_{t\geq 1}$ by using the comparison with percolation.
\end{remark}

 \begin{corollary}  \label{corr:cC}
 	Suppose that the conditions (C1) and (C2) hold and  $\lambda > \lambda_1$ with $\lambda_1$ as in Proposition \ref{prop:theta}. Then there exist $C, \epsilon_1>0$, such that 
 	\begin{align} \label{seclu}
 	&\Prob \left( \exists \, \cC \textrm{ a connected component of } G(\cP)|_{B_o(t)}:   \cC \cap \cC_{\delta}(B_o(t)) = \varnothing, |\cC| \geq t^{d(1-\epsilon_1)}  \right) \notag \\
 	& \leq  Ct^{-(3d+\epsilon_1)},
 	\end{align}
 and 
    \begin{equation} \label{cdcc}
    \Prob \left( \cC_{\delta}(B_o(t)) \subset \cC(B_o(t)) \subset \cC(\R^d)   \right) \geq 1- Ct^{-(3d+\epsilon_1)}.
    \end{equation}	
 \end{corollary}
\begin{proof}[Proof of Corollary \ref{corr:cC}]
Let $\epsilon$ be a small positive constant. We first observe that 
  \begin{equation*} 
  \Prob\left( \exists \, x \in \cP \cap B_o(t):  |B_x(t^{1-2 \epsilon}) \cap \cP| \geq t^{d(1-\epsilon)} \right) \leq \exp(-ct^{d(1-2 \epsilon)}),
  \end{equation*}
 for some $c>0$. Moreover, for any connected set $\cC$, if $d_{\infty}(\cC) = \max_{x, y \in \cC} |x-y|_{\infty} \leq t^{1-2 \epsilon}$ then  $\cC \subset B_z(t^{1-2\epsilon})$ for any $z \in \cC$. Thus 
   \begin{align} \label{ld:num}
   &\Prob \left( \exists \, \cC \textrm{ a connected component of } G(\cP)|_{B_o(t)}:    |\cC| \geq t^{d(1-\epsilon_1)}, \, d_{\infty}(\cC)   \leq t^{1-2 \epsilon} \right) \notag \\
   & \leq \exp(-ct^{d(1-2 \epsilon)}).
   \end{align}
  Now suppose that $\cC \subset B_o(t)$ is a connected set satisfying $d_{\infty}(\cC) > t^{1-2\epsilon}$ and $\cC \cap \cC_{\delta}(B_o(t)) =\varnothing$. Then there exists $\gamma =(x_i)_{i=0}^l \subset B_o(t) \setminus \cC_{\delta}(B_o(t))$ such that
$x_i \sim x_{i+1}$ for all $i=0,\ldots,l-1$ and $|x_0-x_l|_{\infty} > t^{1-2 \epsilon}$. 

We divide the cube $B_o(t)$ into cubes of size $t^{1-2 \epsilon}/16$, and call the center of these cubes by $(y_i)_{i=1}^L$ with $L \asymp t^{2d \epsilon}$.  Then there exists an index $i_0$ such that $x_0 \in B_{y_{i_0}}(t^{1-2 \epsilon}/16)$ and thus if $\cC_{\delta}(B_{y_{i_0}}(t^{1-2\epsilon})) \subset \cC_{\delta}(B_o(t))$ then $\gamma =(x_i)_{i=0}^l$ is a realization of $\cA_{\theta_{[d]}}(y_{i_0}, \tfrac{1}{16}, t^{1-2\epsilon})$. In the other words,
  \begin{align*}
  &&\{ \exists \, \cC \textrm{ a connected component of } G(\cP)|_{B_o(t)}: d_{\infty} (\cC) >t^{1-2 \epsilon}  \} \cap \cE_1 \subset \cE_2,
  \end{align*}
where
  \[ \cE_1 = \cap_{i=1}^L \{ \cC_{\delta}(B_{y_{i}}(t^{1-2\epsilon})) \subset \cC_{\delta}(B_o(t)) \}, \quad \cE_2 = \cup_{i=1}^L \cA_{\theta_{[d]}}(y_{i}, \tfrac{1}{16}, t^{1-2\epsilon}). \]
Thus
  \begin{align*}
  &\Prob \left( \exists \, \cC \textrm{ a connected component of } G(\cP)|_{B_o(t)}: d_{\infty} (\cC) >t^{1-2 \epsilon} \right) \\
   &\leq  \Prob(\cE_2) + \Prob(\cE_1^c)   \leq L \left( \nu (t^{1-2\epsilon}, t) + \theta_d(\tfrac{1}{16}, t^{1-2 \epsilon}) \right) \\
   &\leq C t^{-(3d + \epsilon_0) (1-2 \epsilon) + 2d \epsilon} \leq C t^{-(3d + \epsilon_0/2)},
  \end{align*} 
  for $\epsilon$ small enough. Combining this with \eqref{ld:num} we obtain \eqref{seclu}. 
  
We turn to prove \eqref{cdcc}.  By \eqref{seclu}, with probability $1-\cO(t^{-3d+\epsilon_1})$, all the connected components in $B_o(t)$ that are not intersected with $\cC_{\delta}(B_0(t))$ have size smaller than $t^{d(1-\epsilon_1)}$. On the other hand, by Lemma \ref{lem:comp} (i),   $|\cC_{\delta}(B_o(t))| \geq c t^d$ and $\cC_{\delta}(B_o(t)) \subset \cC_{\delta}(\R^d) \subset \cC(\R^d)$ with probability $1-\exp(-ct)$ for some $c>0$. Altogether gives the proof of \eqref{cdcc}. 
  \end{proof}

 \subsection{Weak stabilization and moment condtion} 
 \label{ssec:1-2}

\subsubsection{Proof of the weak stabilization \eqref{delta-infinity} } \label{sss:d-i}

Let $\{W_n\}$ be an increasing sequence of cubes tending to $\R^d$. Let $A$ be the event that the vertex $\zero$ is connected to the infinite cluster $\cC(\R^d)$. Recall the expression of $\Delta_n =  \Delta(W_n)$ in three different cases in the proof of Theorem~\ref{thm:size}. Recall also the definition of the limit add-one cost 
\[
	\Delta = \I(A) \Big(1 + \# \{x \in \cP : x \leftrightarrow \zero , x \not \in \cC(\R^d)\} \Big).
\]

Define 
\[
	\Delta_n' = \I(A) \Big(1 + \# \{x \in \cP|_{W_n} : x \overset{W_n}\leftrightarrow \zero , x \not \in \cC(\R^d)\} \Big).
\]
It is clear that with probability one,  $\Delta_n' \to \Delta$, and thus, $\Delta_n' \Pto \Delta$ as $n \to \infty$. Therefore, for the weak stabilization, it suffices to show that 
\begin{equation}
	\Delta_n - \Delta_n' \Pto 0 \quad \text{as} \quad n \to \infty.
\end{equation}

Denote by $A_n$ the event that 
\[
	A_n = \{\cC(W_n) \subset \cC(\R^d)\}.
\]
By Corollary~\ref{corr:cC}, $\Prob(A_n) \to 1$ as $n \to \infty$. On the event $A$, denote by $x_1, \dots, x_k$  the points in the infinite cluster $\cC(\R^d)$ directly connected to $\zero$. Let $B_n$ be the event that 
\[
	B_n = A \cap \{x_i \in \cC(W_n), i = 1, \dots, k\} \subset A_n.
\] 
Then we claim that 
\begin{equation}\label{Bn}
	\Prob(B_n) \to \Prob(A) \quad \text{as} \quad n \to \infty.
\end{equation}
Indeed, for $m > n$, note that  
\[
	A \cap A_n  \cap \{\cC(W_n) \subset \cC(W_m) \} \cap \{x_i \overset{W_m}\leftrightarrow \cC(W_n), i = 1, \dots, k\} \subset B_m,
\]
which implies 
\begin{align*}
	\Prob(B_m) &\ge \Prob(A) - \Prob(A_n^c) - \Prob(\cC(W_n) \not \subset \cC(W_m)) \\
	&\quad - \Big(1 -  \Prob(x_i \overset{W_m}\leftrightarrow \cC(W_n),  i = 1, \dots, k)\Big).
\end{align*}
For fixed $n$, it is clear that the second term and the fourth term go to zero as $m \to \infty$. The third term also goes to zero as $n, m \to \infty$ by taking into account of the equation~\eqref{cdcc} and Lemma~\ref{lem:comp}(i). Taking the liminf in the above equation for fixed $n$, then letting $n$ tend to infinity, we get that 
\[
	\Prob(A) \ge \liminf_{m \to \infty } \Prob(B_m) \ge \Prob(A),
\]
proving the claim.

It follows from the definition of  the event $B_n$ that on $B_n$,
\begin{align*}
	 \Delta_n &= 1 + \# \{x \in \cP|_{W_n} : x \overset{W_n} \leftrightarrow \zero, x \not \in  \cC(W_n)\} \\
	 &= 1 + \# \{x \in \cP_{W_n} : x \leftrightarrow \zero, x \not \in \cC(\R^d)\} = \Delta_n'.
\end{align*}
Now let us write 
\begin{align*}
	\Delta_n - \Delta_n' &= \Delta_n \I(B_n) - \Delta_n' + \Delta_n  (\I(A ) - \I(B_n)  ) + \Delta_n \I(A^c) \\
	&= \Delta_n' (\I(B_n) - \I(A)) +  \Delta_n  (\I(A ) - \I(B_n)  ) + \Delta_n \I(A^c).
\end{align*}
The first and the second terms converge in probability to zero by the claim~\eqref{Bn}. It remains to show that the third term converges to zero in probability. But it is an easy consequence of the fact that on the event $A^c$, the finite component containing the vertex $\zero$ is of course smaller than the biggest component when $n$ is large enough. And thus, when $n$ is large enough, $\Delta_n = 0$ (Case 3 in the expression of $\Delta_n$). The proof of the weak stabilization is complete.

\subsubsection{Proof of the moment estimate \eqref{edt3}}  
Let $W \ni \zero$ be a cube. Recall the following upper bound for the add-one cost
\begin{align*}
	0 \le \Delta(W) &\le 1 + \#\{x \in \cP|_W : x \overset{W}\leftrightarrow \zero, x \not \in \cC(W) \} \\
	&=1 + \sum_{x \in \cP|_W } \I (x \overset{W}\leftrightarrow \zero, x \not \in \cC(W)). 
\end{align*}
Thus 
\begin{align}
	\Delta(W)^3 &\le 4 + 4 \bigg(  \sum_{x \in \cP|_W } \I (x \overset{W}\leftrightarrow \zero, x \not \in \cC(W)) \bigg)^3 \notag\\
	&= 4 + 24 \sum_{\{x, y, z\} \subset \cP|_W} \I (x\overset{W}\leftrightarrow \zero, y\overset{W}\leftrightarrow \zero, z\overset{W}\leftrightarrow \zero, x, y, z \not \in \cC(W))  \notag\\
	&+ 24 \sum_{\{x, y\} \subset \cP|_W}  \I (x\overset{W}\leftrightarrow \zero, y\overset{W}\leftrightarrow \zero,  x, y \not \in \cC(W)) + 4 \sum_{x \in \cP|_W} \I (x \overset{W}\leftrightarrow \zero, x \not \in \cC(W)) . 	\label{Delta3}
\end{align}
Here the first and the second sums are taken over subsets of three elements, and two elements of $\cP|_{W}$, respectively.

For $x \in \R^d$ and a finite subset $\cX \subset \R^d$ with  $\zero \not \in \cX$, let $G(\cX)$ be a random graph generated by connecting any two points $(y, z)$ of $\cX$ with probability $\varphi(y - z)$, and $G(\cX \cup \{x\})$, when $x \not \in \cX$, be the random graph obtained from $G(\cX)$ by adding the vertex $x$ and new edges connected to $x$ (independently with probability $\varphi(x - y), y \in \cX$). Similarly, let $G(\cX \cup \{x, \zero\})$ be the random graph obtained from $G(\cX \cup \{x\})$ by adding the vertex $\zero$ and new edges from $\zero$. Denote by $\cC(G)$ the biggest connected component of a graph $G$. Define $\xi(x; \cX)$ to be the probability that $x$ and $\zero$ are in the random graph  $G(\cX \cup \{x, \zero\})$  and $x$ does not belong to the biggest component of $G(\cX \cup \{x\})$,
\[
	\xi(x; \cX) = \Prob_2(x \leftrightarrow \zero, x \not \in \cC(G(\cX \cup \{x\}))).
\] 
Here we may consider $\Prob = \Prob_1 \otimes \Prob_2$, where $\Prob_1$ and $\Prob_2$ are the probability measures for the Poisson process and for connecting edges, respectively.
With those notations, by the Mecke formula, the expectation of the last sum in the estimate~\eqref{Delta3} can be written as 
\begin{align}
	\Ex \bigg[ \sum_{x \in \cP|_W} \I (x \overset{W}\leftrightarrow \zero, x \not \in \cC(W)) \bigg] &= \Ex_1 \bigg[  \sum_{x \in \cP|_W} \xi(x; \cP|_W)  \bigg] \notag \\
	&=
	 \lambda \int_W  \Ex_1 [\xi(x; \cP|_W)] dx.\label{sum3}
\end{align}

Note that $\Ex_1 [\xi(x; \cP|_W)]$ is nothing but the probability of the event that $x$ is connected to $\zero$ in $G(\cP|_W \cup \{x, \zero\})$ and $x$ does not belong to the biggest component of $\cC(G(\cP|_W \cup \{x\}) )$. The latter condition implies that $\cC(G(\cP|_W \cup \{x\})) = \cC(G(\cP|_W)) = \cC(W)$, because $G(\cP|_W)$ is a subgraph of $G(\cP|_W \cup \{x\})$. 
Therefore, we deduce that 
\begin{equation}\label{xi1}
	\Ex_1 [\xi(x; \cP|_W)] \le    \Prob(\cA_x),
\end{equation}
where
\begin{align*}
   \cA_{x} = \{ \exists (x_i)_{i = 0}^l  \in \cP|_W \cup \{x, \zero\} \setminus \cC(W) : x_i \sim x_{i + 1}, i = 0, \dots, l - 1, (x_0 = \zero, x_{l} = x) \}.
\end{align*}

\begin{lemma}\label{prop:expdecay}
 Assume that $\lambda > \lambda_1$ with $\lambda_1$ as in Proposition \ref{prop:theta}. Then there exist positive constants $C_2, \epsilon_2$, such that
 \begin{equation}\label{Ax}
 \Prob(  \cA_{x}) \leq C_2 |x|_{\infty}^{-(3d+\epsilon_2)}.
 \end{equation}
\end{lemma}

\begin{proof}[Proof of the moment condition~\eqref{edt3}]

The expectation of the third sum in the estimate~\eqref{Delta3} is uniformly bounded by combining equations~\eqref{sum3}--\eqref{Ax}. Let us now show the uniform boundedness of the expectation of the first sum.

For $x, y , z \in W$, and a finite set $\cX \subset W$, we build random graphs in the same way as above  in the order that 
\[
	G(\cX) \to G(\cX \cup \{x\}) \to G(\cX \cup \{x, y \}) \to G(\cX \cup \{x, y , z\}) \to G(\cX \cup \{x, y, z, \zero\}).
\]
Define 
\[
	\xi_3(x, y, z; \cX) = \Prob_2(x \leftrightarrow \zero, y \leftrightarrow \zero, z \leftrightarrow \zero,  \text{and } x, y , z \not \in \cC(G(\cX \cup \{x, y, z\}))).
\]
Then by the multivariate Mecke equation (Theorem~4.4 in \cite{Last-Penrose-book}), the expectation of the first sum is written as
\begin{align*}
	&\Ex \bigg[\sum_{\{x, y, z\} \subset \cP|_W} \I (x\overset{W}\leftrightarrow \zero, y\overset{W}\leftrightarrow \zero, z\overset{W}\leftrightarrow \zero, x, y, z \not \in \cC(W))   \bigg] \\
	&= \frac{\lambda^3}{6} \int_{(W)^3} \Ex_1[\xi_3(x, y, z; \cP|_W)] \, dx dy dz.
\end{align*}
Note that $ \Ex_1[\xi_3(x, y, z; \cP|_W)] $ is the probability of the event $\cA_{x, y, z}$ that $x, y $ and $z$ are connected to $\zero$ in the random graph $G(\cP|_W \cup \{x, y, z, \zero\})$ and $x, y$ and $z$ do not belong to the biggest component of $G(\cP|_W \cup \{x, y, z\})$. The latter condition implies that 
\[
	\cC(G(\cP|_W \cup \{x, y, z\})) = \cC(G(\cP|_W \cup \{x, y\})) = \cC(G(\cP|_W \cup \{x\})) = \cC(W).
\]

Note that $\cA_{x, y, z}$ is not included in $\cA_x \cap \cA_y \cap \cA_z$. However, it holds that 
\[
	\cA_{x, y, z}	\subset  \cA_x' \cap \cA_y' \cap \cA_z',
\]
where $\cA_x', \cA_y'$ and $\cA_z'$ are events obtained by replacing the condition 
\[
	\exists (x_i)_{i = 0}^l  \in \cP|_W \cup \{x, \zero\} 
\]
in the definition of $\cA_x, \cA_y$ and $\cA_z$, respectively, to the condition that 
\[
		\exists (x_i)_{i = 0}^l  \in \cP|_W \cup \{x, y, z, \zero\}. 
\]
We can show that the probabilities of $\cA_x', \cA_y'$ and $\cA_z'$ also satisfy analogous  estimates as those for $\cA_x, \cA_y$ and $\cA_z$ in Lemma~\ref{prop:expdecay}, that is, there exist positive constants $C_3, \epsilon_3$, such that
\begin{equation}
 \Prob(  \cA_{u}') \leq C_3 |u|_{\infty}^{-(3d+\epsilon_3)}, \quad u \in \{x, y, z\}.
\end{equation}
Together with the following inequality
\[
	\Prob(\cA_x' \cap \cA_y' \cap \cA_z') \le \min\{\Prob(\cA_x'), \Prob(\cA_y'), \Prob(\cA_z')\} \le \Prob(\cA_x')^{1/3}\Prob(\cA_y')^{1/3}\Prob(\cA_y')^{1/3},
\]
we deduce that 
\[
	\Ex_1[\xi_3(x, y, z; \cP|_W)] = \Prob(\cA_{x, y, z} )  \le \max\{1, C |x|^{-d + \epsilon} |y|^{-d + \epsilon} |z|^{-d + \epsilon}\} =: M(x, y, z), 
\]
for some constants $C$ and $\epsilon$. Note that the function $M$ is integrable over $(\R^d)^3$. Therefore,
\begin{align*}
	\int_{(W)^3} \Ex_1[\xi_3(x, y, z; \cP|_W)] \, dx dy dz & \le \int_{(W)^3} M(x, y, z)\, dx dy dz 	\\
	&\le\int_{(\R^d)^3} M(x, y, z) \, dx dy dz < \infty,
\end{align*}
implying the uniform boundedness of the expectation of the first sum follows. Similar argument yields the uniform boundedness of the expectation of the second sum in the estimate~\eqref{Delta3}. The moment condition \eqref{edt3} is proved. 
\end{proof}

\begin{proof}[Proof of Lemma \ref{prop:expdecay}] 
Let $W$ be a cube containing the vertex $\zero$. Let $u=(u_1,\ldots,u_d)$ be the point such that $W$ has the expression $W=\prod_{j=1}^d [u_j,u_j + h]$ with $h$ the size of $W$.  Since $o \in W$, we have $u_j \leq 0$ for all $j=1,\ldots,d$. Without loss the generality, we assume that $o$ is in the lowest conner of $W$, that is, $o \in \prod_{j=1}^d [u_j,u_j+\tfrac{h}{2}]$.

For each $J\subset [d]$, we denote by $u^J$ the vertex such that $u^J_j=0$ for $j \in J$ and $u^J_j=u_j$ for $j \in [d]\setminus J$. We claim that there exist  positive constants $c$ and $\{c_J,  J \subset [d] \}$,  such that for all $0 <t \leq h$, one has $ B^J_{u^J}(ct) \subset W$ for all $J \subset [d]$ and
 \begin{align} \label{bjzj}
 & \Big\{ \exists \,  (x_i)_{i=0}^l \subset W: x_0=o, |x_l|_{\infty}\geq t,  x_i \sim x_{i-1} \, \forall \, i=1,\ldots, l \Big \} \cap \cA_{\kappa} (u, c, t)^c \notag \\
 & \qquad \subset \bigcup_{J\subset [d]} \cA'_J(u^J,c_Jt),
\end{align}
where $\cA_{\kappa}(u, c, t)$ is defined as in \eqref{def-akp0} (including two more points $\{x, \zero\}$) and 
\begin{align*}
  \cA'_J(x,s)&= \Big \{    \exists \, \gamma = (y_i)_{i=0}^l \subset \cP:  y_i \in B^J_x(s), y_i \sim y_{i+1} \, \forall \, i=0,\ldots, l-1,  \notag \\   
  & \qquad  y_0 \in B^J_x(\tfrac{s}{16}), \,  y_l \in B^J_x(2s)  \setminus B^J_x(s) \Big \},
 \end{align*}
with $B_x^J(s)$ defined as in \eqref{bxjr}.
 
Assuming this claim for a moment, we return to estimate $\Prob(A_x)$. By Lemma~\ref{lem:comp}(i), Corollary \ref{corr:cC} and  the estimate \eqref{pb-akp},
 \begin{equation} \label{pce}
 \Prob(\cE^c) \leq  C|x|^{-(3d+\epsilon)}, 
 \end{equation} 
 for some $C, \epsilon>0$, where 
  \begin{equation*}
   \cE: = \Big\{ \cC_{\delta}(B^J_{u^J}(c_J|x|_{\infty})) \subset \cC_{\delta}(W) \subset \cC(W) \, \forall \, J \subset [d] \Big\} \cap \cA_{\kappa}(u, c, |x|_{\infty})^c.
  \end{equation*}
 Suppose that $\cA_{x}  \cap \cE$  happens. Then there exists  $\gamma=(x_i)_{i=0}^l \subset W \setminus \cC_{\delta}(W)$ such that $x_0=o, x_l =x, x_i \sim x_{i-1}$ for $i=1,\ldots, l$.   Hence, using \eqref{bjzj} we obtain that on $\cE$, there exists $J\subset [d]$, such that  the event $\cA_{\theta_{J}}(u^J, \tfrac{1}{16}, c_J|x|_{\infty})$ happens. Therefore, using Proposition \ref{prop:theta} we have
  \begin{equation*}
  \Prob(\cA_{x}  \cap \cE) \leq 2^d \max_{0 \leq j \leq d}  \theta_j(\tfrac{1}{16},c|x|_{\infty})\leq C|x|_{\infty}^{-(3d+\epsilon)},
  \end{equation*}
  for some $C,\epsilon >0$. Combining this with  \eqref{pce}, we obtain the desired estimate~\eqref{Ax}. 
  
 Now we show the relation \eqref{bjzj} for $d=2$, the proof for $d\geq 3$ is similar and hence is omitted. Here, we have $o\in u+[0,\tfrac{h}{2}]^2$ and $0<t\leq h$. We take $c=2^{-9}$. There are four cases corresponding to the relative position of $o$ and $u$ as follows.  
 
 First, if $o-u \in [0,\tfrac{t}{2^4}]^2$,  we consider $J=\varnothing$ and $c_{\varnothing}=2^{-1}$, $u^{\varnothing} =u$. Define $m=\inf \{i: x_i \notin  B^{\varnothing}_{u^\varnothing}(c_\varnothing t)  \}$ (recall that $ B^{\varnothing}_{u^\varnothing}(s) = u^\varnothing+ [0,2s]^2$). We have $x_m \in B^{\varnothing}_{u^\varnothing}(2c_\varnothing t)$, as $x_{m-1} \in B^{\varnothing}_{u^\varnothing}(c_\varnothing t)$ and $|x_m-x_{m-1}|_{\infty} \leq c t$ by $\cA_{\kappa}(u, c, t)^c$. Then since 
  $$o \in B^\varnothing_{u^\varnothing}(\tfrac{c_\varnothing t}{2^4}) \subset B^\varnothing_{u^\varnothing}(c_\varnothing t) \subset W \cap [-t,t]^2,$$  the path $(x_i)_{i=0}^{m}$ is a realization of $\cA'_{\varnothing}(u^{\varnothing},c_\varnothing t)$.
 
 Second, if $o-u \in [\tfrac{t}{2^4},\tfrac{h}{2}] \times [0,\tfrac{t}{2^8}]$, then consider $J=\{1\}, u^{\{1\}}=(0,u_2),  c_{\{1\}}=2^{-4}$ and observe that 
 $$o \in B^{\{1\}}_{u^{\{1\}}}(\tfrac{c_{\{1\}}t}{2^4}) \subset B^{\{1\}}_{u^{\{1\}}}(c_{\{1\}} t) \subset W \cap [-t,t]^2.$$
 Then by the same argument as in the first case, we have  $\cA'_{\{1\}}(u^{\{1\}},c_{\{1\}} t)$ happens. 
 
 Third, if $o-u \in [0,\tfrac{t}{2^8}] \times [\tfrac{t}{2^4},\tfrac{h}{2}]$, using the same argument as above, we have $\cA'_{\{2\}}(u^{\{2\}},c_{\{2\}} t)$ occurs with $c_{\{2\}}=2^{-4}$.
 
 Finally, if the above three cases do not hold, then $u_1,u_2 \leq -\tfrac{t}{2^8}$ and hence
  $$o \in B^{[2]}_{u^{[2]}}(\tfrac{c_{[2]}t}{2^4}) \subset B^{[2]}_{u^{[2]}}(c_{[2]} t) \subset W \cap [-t,t]^2,$$
 with  $u^{[2]}=o, c_{[2]} =2^{-8}$. We then can conclude that $\cA'_{[2]}(u^{[2]},c_{[2]} t)$ happens. The proof of \eqref{bjzj} is completed.
\end{proof}

\subsection{Proof of Proposition \ref{prop:theta}} \label{ssec:tht}
For any $t \geq 2s>0$,  $\alpha \in (0,1/4)$, and $x \in \R^d$, recall $\beta(t)$ and $\nu(s,t)$ and $\kappa(\alpha,t)$ from Sub-section~\ref{ssec:renor}.

The key to the proof of Proposition \ref{prop:theta} is the following recursive relation, which is inspired by the ideas in the study of Boolean percolation in \cite{DRT, Go}.
\begin{lemma} [Renormalization estimate] \label{lem:renorm}
	There exists  a positive constant $K\geq 1$,  such that 
	\[ \theta(\tfrac{1}{16},t) \leq \kappa (\tfrac{1}{2^{10}},t) + K\nu(\tfrac{t}{2^7},t) + K \left(\theta(\tfrac{1}{16}, \tfrac{t}{2^3})+\theta(\tfrac{1}{16}, \tfrac{t}{2^7})\right)^2. \]
\end{lemma}
\begin{proof}   We have to prove that for al $0 \leq j \leq d$,
	\begin{equation} \label{receq}
	 \theta_j(\tfrac{1}{16},t) \leq \kappa (\tfrac{1}{2^{10}},t) + K\nu(\tfrac{t}{2^7},t) + K \left(\theta(\tfrac{1}{16}, \tfrac{t}{2^3})+\theta(\tfrac{1}{16}, \tfrac{t}{2^7})\right)^2.
	\end{equation}
For simplicity, we prove here the case $j=0$ and $d=2$ because the proof of general cases is essential the same.

We recall
 \begin{equation}
\theta_0(\tfrac{1}{16},t) = \Prob(\cA_{\theta_{\varnothing}}(o,\tfrac{t}{16},t)),
\end{equation}  
where $\cA_{\theta_J}(s,t)$ and $B^J(x,r)$ are defined as in \eqref{ajst} and \eqref{bxjr}.
  
	We call $A=(2t,0), B=(2t,2t), C=(0,2t)$ and $A_1=(\tfrac{t}{2},0), B_1=(\tfrac{t}{2},\tfrac{t}{2}), C_1=(0,\tfrac{t}{2})$ and $A_2=(\tfrac{3t}{2},0), B_2=(\tfrac{3t}{2},\tfrac{3t}{2}), C_2=(0,\tfrac{3t}{2})$. Then we cover the segments $A_1B_1 \cup B_1C_1$ by squares $S_{1,a}, S_{1,c}, S_{1,1}, \ldots,S_{1,K_1}$, where $S_{1,a}, S_{1,c}$ has the length size $t/2^6$ and are adjacent to $OA, OC$ respectively, and $S_{1,1},\ldots,S_{1,K_1}$ have the length size $t/2^{10}$. Similarly, we cover the segments $A_2B_2 \cup B_2C_2$ by squares $S_{2,a}, S_{2,c},$ $S_{2,1}, \ldots,S_{2,K_2}$. See Figure~\ref{OABC} for an illustration of the cover. 
\begin{figure}[ht]
	\centering
	\includegraphics[width = 11cm]{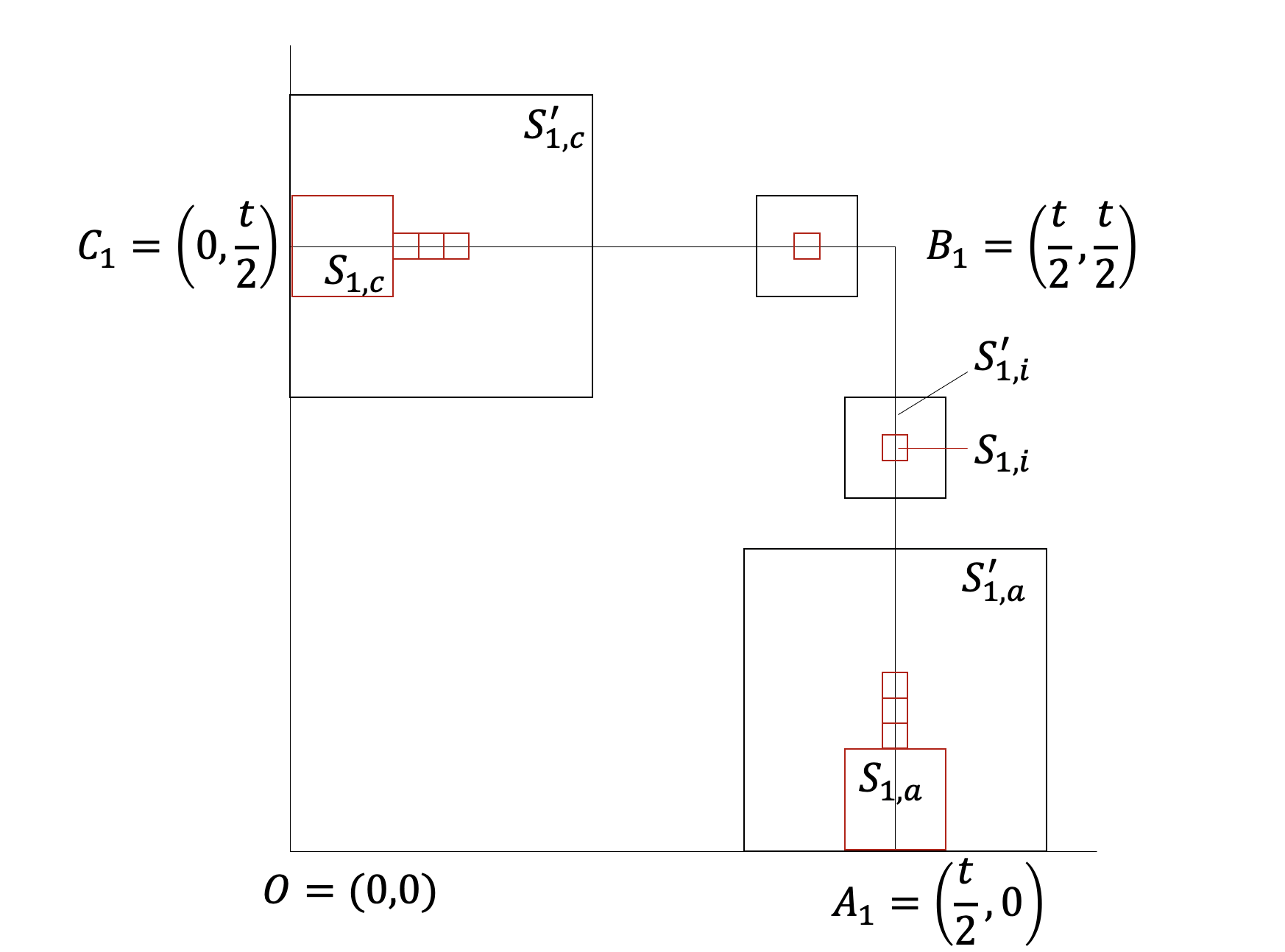}
	\caption{Illustration of the cover of $A_1B_1 \cup B_1 C_1$}\label{OABC}
\end{figure}
Notice that for $i=1,2$,
	\begin{equation}
	S_{i,a} = B_{A_i}^{\{1\}}(\tfrac{t}{2^7}), \quad S_{i,c} = B_{C_i}^{\{2\}}(\tfrac{t}{2^7}),  \quad S'_{i,a}, S'_{i,c} \subset S,
	\end{equation}
where 
  \begin{equation}
 S=[0,2t]^2, \quad  S'_{i,a} := B_{A_i}^{\{1\}}(\tfrac{t}{2^3}), \quad S'_{i,c} := B_{C_i}^{\{1\}}(\tfrac{t}{2^3}).
  \end{equation}
Similarly,  if we call $y_{i,j}$ with $i=1, 2$ and $1 \leq j \leq K_i$, the center of the squares $S_{i,j}$, then 
  \begin{equation}
S_{i,j}:= B_{y_{i,j}}(\tfrac{t}{2^{11}}), \quad S'_{i,j}:= B_{y_{i,j}}(\tfrac{t}{2^7}) \subset S.
  \end{equation}
Define
\begin{align*}
	 \cA_{\nu} (t) &=& \{ \cC_{\delta}(S'_{i,a}), \cC_{\delta}(S'_{i,c}), \cC_{\delta}(S'_{i,j}) \subset \cC_{\delta}(S), \, \forall \,  i=1,2; 1 \leq j \leq K_i \}.
	\end{align*}
We then claim that 
 \begin{equation} \label{claimA}
 \cA_{\theta_{\varnothing}}(o,\tfrac{t}{16},t) \cap \cA_{\kappa}(o,2^{-10},t)^c \cap \cA_{\nu}(t)  \subset \cB_1(t) \cap \cB_2(t),
 \end{equation}
 where  $\cA_{\kappa}(o, 2^{-10},t)$ is defined as \eqref{def-akp0}, and for $i=1,2$,
   \begin{align} 
   \cB_i(t) = \cA_{\theta_{\{2\}}}(A_i,\tfrac{t}{2^7},\tfrac{t}{2^3}) \cup \cA_{\theta_{\{1\}}}(C_i,\tfrac{t}{2^7},\tfrac{t}{2^3}) \bigcup_{j=1}^{K_i} \cA_{\theta_{[2]}}(y_{i,j},\tfrac{t}{2^{11}},\tfrac{t}{2^7}).
   \end{align}     
 Assuming this claim for a moment, we prove the lemma. Notice that  $\cB_1(t)$ depends only on the configuration of the graph inside $[0,\tfrac{3t}{4}]^2$, whereas $\cB_2(t)$ is measurable to the  the configuration of the graph in $ [0,\tfrac{7t}{4}]^2 \setminus [0,\tfrac{5t}{4}]^2$. Hence, the two events are independent, and thus \eqref{claimA} gives that
    \begin{align*}
    \theta_0(\tfrac{1}{16},t) &=\Prob(\cA_{\theta_{[2]}}(o,\tfrac{t}{16},t) ) \leq \Prob (\cA_{\kappa}(o,2^{-10},t)) + \Prob (\cA^c_{\nu}(t)) + \Prob(\cB_1(t)) \Prob(\cB_2(t))\\
    &\leq \kappa(\tfrac{1}{2^{10}},t) + (K_1+K_2+4) \nu(\tfrac{t}{2^7},t) +   \Prob(\cB_1(t)) \Prob(\cB_2(t)).
    \end{align*}    
Moreover, by the union bound, for $i=1,2$,
  \begin{equation}
  \Prob(\cB_i(t)) \leq (2\theta_1(\tfrac{1}{16},\tfrac{t}{2^3})+K_i\theta_2(\tfrac{1}{16},\tfrac{t}{2^7})).
  \end{equation}
 Combining the last two estimates, we obtain \eqref{receq} with $K=(K_1+2)(K_2+2)$.
 
  Now it remains to show \eqref{claimA}. Suppose that $\cA_{\theta_{\varnothing}}(o,\tfrac{1}{16},t) \cap \cA_{\kappa}(t) \cap \cA_{\nu}(t)$ happens. Then there exists a path $\gamma=(x_i)_{i=0}^l$ such that $x_0 \in [0,\tfrac{t}{2^3}]^2$ and   $x_i \in [0,2t]^2 \setminus \cC_{\delta}([0,2t]^2)$ for $i=0,\ldots, l-1$ and $x_l \notin [0,2t]^2$ with positive coordinates. 
  
Since $S_{1,a}, S_{1,c},  S_{1,j}, j=1,\ldots,K_1$ is a cover of $A_1B_1 \cup B_1C_1$ and by the assumption on $\cA_{\kappa}(o,2^{-10}, t)^c$, $|x_i-x_{i+1}|_{\infty} \leq 2^{-10}t$ for all $i=0,\ldots,l-1$,   there must be some vertices of the path $\gamma$ lying in these cubes. So we can define
        \[ l_1 = \min \bigg\{ i \in [1,l-1]: x_i  \in S_{1,a} \cup S_{1,c} \cup  \bigcup_{j=1}^{K_1} S_{1,j} \bigg\}.\]
Suppose that $x_{l_1} \in S_{1,a}$.         
 Define also
  \[ l_1' = \min \{i \geq l_1: x_i \notin S_{1,a}' \}. \]   
 Then by the definition,
   \[ x_{l_1} \in S_{1,a} =B^{\{1\}}_{A_1} (\tfrac{t}{2^7}), \quad x_i \in S_{1,a}' = B^{\{1\}}_{A_1} (\tfrac{t}{2^3}) \, \, \forall i= l_1+1, \ldots, l_1'-1, \]
and  since $\cA_{\kappa}(o, 2^{-10},t)^c$ holds,
 $$x_{l_1'} \in B^{\{1\}}_{A_1}(\tfrac{t}{2^3}+\tfrac{t}{2^{10}}) \setminus B^{\{1\}}_{A_1} (\tfrac{t}{2^3})  \subset B^{\{1\}}_{A_1}(\tfrac{t}{2^2}) \setminus B^{\{1\}}_{A_1} (\tfrac{t}{2^3}).$$ 
   Moreover, we notice that  by $\cA_{\nu}(t)$,
   $$x_i \notin \cC_{\delta}(S'_{1,a}) \, \, \forall i=l_1, \ldots, l_1'-1.$$ 
 Hence,  $(x_i)_{i=l_1}^{l_1'}$ is   a realization for  the event   $  \cA_{\theta_{\{2\}}}(A_1,\tfrac{t}{2^7},\tfrac{t}{2^3})$. The cases that $x_{l_1} \in S_{1,c}$ and $x_{l_1} \in S_{1,j}$ for some   $j=1,\ldots,K_1$ can be treated similarly, leading to the realizations of events $\cA_{\theta_{\{1\}}}(C_1,\tfrac{t}{2^7},\tfrac{t}{2^3})$ and $\cA_{\theta_{[2]}}(y_{1,j},\tfrac{t}{2^{11}},\tfrac{t}{2^{7}})$, respectively. In summary, the event $\cB_1$ happens. 
     
    By the same argument, we can also prove that the vent $\cB_2$ happens and the proof of \eqref{claimA} completes.
\end{proof}

\begin{lemma} \label{lem:abt}
	Let $(a_t)_{t\geq 0}, (b_t)_{t\geq 0} \subset \R_+$, $0<\epsilon_1<\epsilon_2<1$, and $K\geq 1$, $t_0 \geq \tfrac{\epsilon_2}{1-\epsilon_2}$ satisfy for all $t\geq t_0$
	\begin{itemize}
		\item [\rm(i)] $ a_{t} \leq b_{t} + K(a_{\epsilon_1t}+a_{\epsilon_2 t})^2, $
		\item[\rm(ii)] $4K(b_{\epsilon_1 t}+b_{\epsilon_2 t})^2 \leq b_t,$ 
		 \item[\rm(iii)] $a_t \leq 2 b_t$, for all $t_0\leq t \leq t_0/\epsilon_1$.	
	\end{itemize}
 Then, $a_t \leq 2 b_t$ for all $t\geq t_0$.
\end{lemma}
\begin{proof}
By assumption (iii), we need to show $a_{t} \leq 2b_t$ for all $t\geq t_0/\epsilon_1$. We prove by induction in $k$ that this claim holds for $[t_0/\epsilon_1 +k, t_{0}/\epsilon_1 +k+1]$. Let $t$ be in this interval. Then
  \begin{align}
  a_t \leq b_{t} + K(a_{\epsilon_1t}+a_{\epsilon_2 t})^2 \leq b_t + K(2b_{\epsilon_1t}+2b_{\epsilon_2 t})^2 \leq 2b_t.
  \end{align} 
 Here, we used (i) and (ii) for the first and third inequalities respectively and for the second one, we used the induction hypothesis with noting that for $t \in [t_0/\epsilon_1 +k, t_{0}/\epsilon_1 +k+1]$, one has $t_0<\epsilon_1 t< \epsilon_2 t< t_0/\epsilon_1 + k$. 
\end{proof}

\begin{proof}[Proof of Proposition \ref{prop:theta}]

Let $K$ be as in Lemma \ref{lem:renorm}. Define 
  $$a_t = \theta(\tfrac{1}{16},t), \quad b'_t=\kappa (\tfrac{1}{2^{10}},t) + K\nu(\tfrac{t}{2^7},t), \quad \epsilon_1=2^{-7}, \epsilon_2=2^{-3}.$$
  Then by Lemma \ref{lem:renorm}, for all $t\geq 0$,
   \begin{equation*} 
   a_{t} \leq b'_{t} + K (a_{\epsilon_1 t} +a_{\epsilon_2t})^2.
   \end{equation*}
Moreover, by \eqref{pb-akp} and Lemma \ref{lem:comp} (i), 
  \begin{equation*}
  b'_t \leq Ct^{-(3d+\epsilon_0)},
  \end{equation*}   
   for some $C>0$ and $\epsilon$ as in (C2). Hence,  
  \begin{equation} \label{abtk}
  a_{t} \leq b_{t} + K (a_{\epsilon_1 t} +a_{\epsilon_2t})^2,
  \end{equation} 
 with $b_t=Ct^{-(3d+\epsilon_0)}$.  There exists  $t_0 =t_0(\epsilon_0,\epsilon_1,\epsilon_2,K,C)>0$, such that for all $t\geq t_0$
  \begin{equation}
  4K (b_{\epsilon_1 t} +b_{\epsilon_2t})^2 \leq b_t.
  \end{equation} 
 It is clear that $\theta(\tfrac{1}{16},t) \leq \Prob(|G(\cP|_{[-t,t]^d})| \neq |\cC_{\delta}([-t,t]^d)|) \rightarrow 0$ as $\lambda \rightarrow \infty$, by Lemma~\ref{lem:comp}(ii). Thus  for all $\lambda$ large enough
  \begin{equation} \label{lbat}
  a_t \leq 2 b_t, \quad \forall \, t_0\leq t \leq t_0/\epsilon_1.
  \end{equation} 
   Combining the last three estimates and   Lemma  \ref{lem:abt}, we get  that for all $t\geq t_0$
  \[ a_t \leq 2b_t \leq 2Ct^{-(3d+\epsilon)}, \]
which completes the proof.
\end{proof}

\appendix
\section{Quenched CLT}
Let us first recall the setting of the quenched CLT. The underlying probability space is written as the product 
\[
(\Omega, \cF, \Prob) = (\Omega_1, \cF_1, \Prob_1) \times (\Omega_2, \cF_2, \Prob_2)
\] for which the first component of $\hat \eta$ is defined on $\Omega_1$, and the second and the third ones are defined on $\Omega_2$, that is,
\[
\hat \eta(\omega) = \{(x(\omega_1), t(\omega_2), M(\omega_2))\}.
\]
We will use $\Ex_i$ and $\Var_i, (i = 1,2)$ to denote the expectation and the variance with respect to $\Prob_i$. Let $W_2$ be the second Wasserstein distance in the space of probability measures on $\R$ having finite second moment
\[
W_2(\mu, \nu)^2 = \inf_{\gamma \in \Gamma(\mu, \nu)} \iint_{\R \times \R} |x - y|^2 d\gamma(x, y),
\]
where $\Gamma(\mu, \nu)$ is the collection of all measures on $\R^2$ having $\mu$ and $\nu$ as marginal distributions. It is known that the convergence of probability measures under $W_2$ is equivalent to the convergence in distribution plus the convergence of the second moment. Moreover, for two mean-zero random variables $X$ and $Y$ defined on the same probability space, it follows directly from the definition of the distance that  
\[
W_2(X, Y)^2 \le \Ex[(X - Y)^2] = \Var[X - Y].
\]
Assume that the functional $f$ satisfies the conditions in Theorem~\ref{thm:main}. 
Let 
\[
Z_n(\omega_1, \omega_2)=\frac{f(T(\hat \eta|_{W_n})) - \Ex_2[f(T(\hat \eta|_{W_n}))]}{\sqrt{n}},
\]
where $W_n := [-n^{1/d}/2, n^{1/d}/2)^d$. Then for each fixed $\omega_1 \in \Omega_1$, $Z_n$ is a random variable on $\Omega_2$ of mean zero.
We assume in addition that the $(2+\delta)$th moment of $Z_n$ is finite for any $n > 0$, that is, for some $\delta > 0$,
\[
\Ex[|Z_n|^{2 + \delta}] < \infty, \quad \text{for all $n$}.
\]
Then there exists $\sigma_q^2 \ge 0$ such that for any $\varepsilon > 0$,
\begin{equation} \label{qclt-zn}
\Prob_1\left( \omega_1 :  W_2(Z_n(\omega_1, \cdot), \Normal(0,\sigma_q^2)) \ge \varepsilon  \right)   \to 0\quad \text{as}\quad n \to \infty.
\end{equation}
In particular, w.h.p. $(Z_n(\omega_1, \cdot))_{n\geq 1}$ converges weakly to $\Normal(0,\sigma_q^2)$.  

Let us prove the above statement. We will use the notations in the proof of Theorem~\ref{thm:main}. Define $Y_{n,L} = Y_{n,L}(\omega_1, \omega_2)$ as
\[ 
Y_{n,L} := \frac{1}{\sqrt n}\sum_{i = 1}^{\ell_n} \Big(   f(T(\hat \eta|_{C_i})) - \Ex_2[f(T(\hat \eta|_{C_i}))] \Big) =: \frac{1}{\sqrt n}\sum_{i = 1}^{\ell_n} f_i.
\]
Then for fixed $\omega_1$, under $\Prob_2$,  $Y_{n, L}$ is a sum of independent random variables with
\[
\Var_2[Y_{n, L}] = \frac{1}{n} \sum_{i = 1}^{\ell_n} \Var_2[f_i].
\]
Note that the sequence $\{\Var_2[f_i]\}_{i = 1}^{\ell_n}$ is i.i.d.\ under $\Prob_1$. Then the strong law of large numbers implies that for almost surely $\omega_1 \in \Omega_1$, as $n \to \infty$,
\[
\Var_2[Y_{n, L}] = \frac{\ell_n}{n} \frac{1}{\ell_n} \sum_{i = 1}^{\ell_n} \Var_2[f_i] \to \frac{\sigma_L'^2}{L}, \quad \sigma_L'^2 := \Ex_1[\Var_2[f_i]].
\]
Similarly, by the finiteness of the $(2+\delta)$th moment, we obtain that for almost surely $\omega_1 \in \Omega_1$, as $n \to \infty$,
\[
\frac{1}{n}\sum_{i = 1}^{\ell_n}  \Ex_2[|f_i|^{2+\delta}] \to \frac{1}{L} \Ex_1[\Ex_2[|f_i|^{2+\delta}]] = \Ex[|f_i|^{2+\delta}] < \infty.
\]
Then for almost surely $\omega_1 \in \Omega_1$ (those $\omega_1$ such that the above two equations hold), by using Lyapunov's central limit theorem (see \cite[Theorem 27.3]{Billingsley}), we obtain that
\begin{equation}\label{CLT-fixed-L}
W_2(Y_{n, L}(\omega_1, \cdot), \Normal(0, \sigma_L'^2/L) ) \to 0 \quad \text{as}\quad n \to \infty.
\end{equation}
Note that $\sigma_L'^2$ may be zero.

Observe that for any random variable $X$ defined on $\Omega$ with finite second moment, 
\begin{equation}\label{conditional-variance}
\Ex_1[\Var_2[X - \Ex_2[X]]] = \Ex_1[\Ex_2[X^2] - \Ex_2[X]^2] \le \Ex[X^2] - \Ex[X]^2 = \Var[X].
\end{equation}
This implies that $\sigma_L'^2 \le \sigma_L^2$, and thus the sequence $\{\sigma_L'^2/L\}$ is bounded. Let $\sigma_q^2$ be a limit of $\{\sigma_L'^2/L\}$, that is, for some subsequence $\{L_k\}$ tending to infinity, 
\[
\sigma_q^2 = \lim_{k \to \infty}  \frac{\sigma_{L_k}'^2} {L_k}.
\]
We are going to show that for this $\sigma_q^2$, the quenched central limit theorem~\eqref{qclt-zn} holds. (And thus $\sigma_q^2$ is unique as a consequence.)  
It follows from the observation~\eqref{conditional-variance} and the estimate~\eqref{CLT-approximation} that 
\begin{equation}\label{approximation-2}
\lim_{L \to \infty} \limsup_{n \to \infty}\Ex_1 \left[ \Var_2\left[ Z_n  - {Y_{n, L}} \right]	 \right]  = 0.
\end{equation}
This is a key estimate to show our result.

Next, by the triangle inequality, we see that 
\begin{align*}
W_2(Z_n(\omega_1, \cdot), \Normal(0, \sigma_q^2)) &\le W_2(Z_n, Y_{n, L} ) + W_2(Y_{n, L} , \Normal(0, \sigma_L'^2 / L)) \\
&\quad + W_2(\Normal(0, \sigma_L'^2 / L),   \Normal(0, \sigma_q^2)).
\end{align*}
Here for simplicity, we have removed $(\omega_1, \cdot)$ in formulae.
Let $\varepsilon > 0$ be given. By the definition of $\sigma_q^2$, when $k$ is large enough, for $L = L_k$,
\[
W_2(\Normal(0, \sigma_L'^2 / L),   \Normal(0, \sigma_q^2)) <\frac \varepsilon 3.
\]
For those $L$, the above triangle inequality implies that
\begin{align*}
&\Prob_1(W_2(Z_n, \Normal(0, \sigma_q^2)) \ge \varepsilon) \\
&\quad  \le \Prob_1\left(W_2(Z_n, Y_{n, L}) \ge \frac\varepsilon 3 \right) + \Prob_1 \left(W_2(Y_{n, L}, \Normal(0, \sigma_L'^2 / L)) \ge \frac\varepsilon 3\right). 
\end{align*}
Since as $n \to \infty$,  the second term goes to zero by the almost sure convergence~\eqref{CLT-fixed-L}, it follows that
\begin{align*}
\limsup_{n \to \infty} \Prob_1(W_2(Z_n, \Normal(0, \sigma_q^2)) \ge \varepsilon) &\le \limsup_{n \to \infty} \Prob_1\left(W_2(Z_n, Y_{n, L}) \ge \frac\varepsilon 3 \right) \\
&\le \limsup_{n \to \infty} \Prob_1\left(\Var_2 [Z_n - Y_{n, L}] \ge \frac{\varepsilon^2} 9 \right) \\
&\le \limsup_{n \to \infty} \frac{9}{\varepsilon^2} \Ex_1[\Var_2 [Z_n - Y_{n, L}]].
\end{align*}
Here we have used the inequality $W_2(X, Y) \le \Var[X-Y]$ for mean zero random variables $X$ and $Y$ defined on the same probability in the second line and Markov's inequality in the last line. The desired result immediately follows from the estimate~\eqref{approximation-2}.  \hfill $\square$

\section{Proof of Lemma \ref{lem:comp}}\label{appendix_proof}
\begin{proof}
For the convenience, let us recall the construction of the connected component $\cC_{\delta}(\Lambda)$. For each $\delta>0$, we tessellate the whole space $\R^d$ to cubes of size $\delta$ and call $\Gamma$  the collection of cubes.   Let $G_{\delta}$ be the random graph obtained from $G(\cP)$ by deleting the edges between vertices in non-adjacent cubes. 
		For each cube $B \in \Gamma$, when $B \cap \cP \neq \varnothing$, we take arbitrarily a point in $B \cap \cP \neq \varnothing$, say $x_B$, to be the representation of $B$.  Let $\textrm{Per}(\delta)$ be the induced subgraph of  $G_{\delta}$ restricted on the vertex set  $V=\{x_B: \, B \in \Gamma,  B \cap \cP =\varnothing \}$. Then for each cube $\Lambda$, we define
		\[  \cC^{{\rm per}}_{\delta} (\Lambda) = \textrm{ the biggest cluster of $\rm{Per}(\delta)|_{\Lambda}$},\]
		and
		\[ \cC_{\delta} (\Lambda) = \textrm{ the connected component of  $G_{\delta}|_{\Lambda}$ containing $\cC^{{\rm per}}_{\delta} (\Lambda)$}. \]

	If we consider each cube $B \in \Gamma$ as a point in $\Z^d$ and then we obtain a percolation. More precisely, each cube $B \in \Gamma$ is called open if $B \cap \cP \neq \varnothing$ and thus 
\[
	p=\Prob(\textrm{a cube is open}) =1-\exp(-\lambda \delta^d).
\]
	For any two adjacent open cubes $B_1$ and $B_2$, we draw an edge between them if there is an edge between their representations. In fact, the probability that the two open cubes are connected is
\[
	\varphi(x_{B_1}-x_{B_2}) \geq  q:= \inf _{x_1 \in B_1, x_2 \in B_2}\varphi (x_1-x_2).
\]
	Then,   $\textrm{Per}(\delta)$ can be viewed as   a bond percolation on $\Z^d$ with the following rule.  An edge $e=\{a,b\}$ is called open if both $a$ and $b$ are open and the edge between $a$ and $b$ is drawn. So $\textrm{Per}(\delta)$ is indeed a locally-dependent percolation (since the statuses of $e=\{a,b\}$ and $f=\{c,d\}$ are independent if $\{a,b\} \cap \{c,d\}=\varnothing$) with parameters 
	\[ p_e = \Prob(e \textrm{ is open}) \geq p^2q, \]
	for all the edges $e$.  	By \cite{LSS},  there exists $p'$, such that whence $p_e > p'$ for all edges $e$,  $\textrm{Per}(\delta)$ stochastically dominate the supercritical standard bond percolation on $\Z^d$. 
	
	By (C1),  we have $q \rightarrow 1$ as $\delta \rightarrow 0$. Moreover, $p\rightarrow 1$ when $\lambda \rightarrow \infty$ and $\delta$ is fixed. Therefore, we can find $\delta_0, \lambda_0$, such that if $\lambda > \lambda_0$, and $\delta=\delta_0$,
	\begin{equation} \label{pepp}
	p_e\geq p^2q \geq \tfrac{1+ p'}{2} > p', 
	\end{equation}
	for all edges $e$,	and thus  $\textrm{Per}(\delta)$ stochastically dominates the supercritical bond percolation. Thus  for all $t>0$
	\begin{align*}
		\beta_{\delta}(t) &= \Prob(\cC_{\delta}(B_o(t)) \not \subset \cC_{\delta}(\R^d)) = \Prob(\cC_{\delta}(B_o(t)) \cap  \cC_{\delta}(\R^d) = \varnothing)\\
		&\leq \Prob(\cC^{{\rm per}}_{\delta}(B_o(t)) \cap \cC^{{\rm per}}_{\delta}(\R^d) =\varnothing)   \leq \exp(-ct),   
	\end{align*}
	for some universal constant $c>0$. Here for the first inequality, we used the fact that ${\rm Per}(\delta)$ is a subgraph of $G_{\delta}$ and for  the last inequality we used  stochastic domination obtained above  and standard estimates in Bernoulli percolation on $\Z^d$, see e.g.\ \cite{G}.  By the same arguments, we can also prove the other estimate in (i). 
	
	We turn to prove (ii). Observe that $|G(\cP|_{\Lambda})|=|\cC_{\delta}(\Lambda)|$ when for any pair of adjacent small cubes, there exists an edge connecting them and all the subgraphs constrained in  cubes are connected. Hence, 
	\begin{align} 
	&\Prob(|G(\cP|_{\Lambda})| \neq |\cC_{\delta}(\Lambda)|) \leq \Prob(\exists B \in \Gamma|_{\Lambda}: G(\cP|_{B}) \textrm{ is not connected}) \notag\\
	& \qquad +\Prob(\exists B, B' \in \Gamma|_{\Lambda}: B, B'\textrm{ are adjacent and there is no edge between them}), \label{GnC}
	\end{align}
	where $\Gamma|_{\Lambda}$ is the set of  cubes that intersect with $\Lambda$. 
	
	Given two adjacent cubes $B$ and $B'$, let $X$ and $X'$ denote the number of vertices of $\cP$ in $B$ and $B'$ respectively. Then $X$ and $X'$ are i.i.d.\ random variables with the Poisson distribution of  mean $\lambda \delta^d$. Notice that each pair $\{x, y\in \cP|_B\}$, or each pair $\{z\in B, z' \in B'\}$ is connected with probability larger than $q \in (0,1]$.
	 Therefore,
	\begin{align*}
		& \Prob(G(\cP|_{B}) \textrm{ is not connected} \mid X) \leq \Prob(\exists \, x, y \in \cP_B: d(x,y) >2 \mid X) \\
		&\leq \Prob(\exists \, x, y \in \cP_B: \forall z \in \cP_B, x\not \sim z \textrm{ or } y \not \sim z \mid X) \\
		&\leq X(X-1)(1- q^2)^{(X-2)},
	\end{align*}
	where $d$ is the graph distance in $G(\cP|_{B})$. Moreover,  
	\begin{align*}
		&& \Prob(\textrm{there is no edge between } B, B' \mid X,X') \leq (1- q)^{XX'}.
	\end{align*}
	Combining the above inequalities with \eqref{GnC}, we obtain 
	\begin{align*}
		\Prob(|G(\cP|_{\Lambda})| \neq |\cC_{\delta}(\Lambda)|) \leq |\Gamma|_{\Lambda}| \E \left[X(X-1)(1-p_1^2)^{(X-2)} \right] + |\Gamma|_{\Lambda}|^2 \E\left[(1-p_2)^{XX'}\right].
	\end{align*}
	The right hand side of the above tends to $0$ as  $\lambda \rightarrow \infty$, since 
	$X$ and $X'$ are i.i.d.\ random variables with law $\textrm{Poi}(\lambda \delta^d)$.
\end{proof}

\bigskip
\noindent\textbf{Acknowledgment.} K.D.T is partially supported by JST CREST Mathematics (15656429) and JSPS KAKENHI Grant Number JP19K14547. The work of V.H.C is supported by NUS Research Grant R-155-000-208-112.

\begin{footnotesize}

\end{footnotesize}

%

\end{document}